\documentclass[11pt]{amsart}

\usepackage{amsmath, amsthm, amssymb, amsfonts, enumerate}
\usepackage[colorlinks=true,linkcolor=blue,urlcolor=blue]{hyperref}
\usepackage{dsfont}
\usepackage{color}
\usepackage{geometry}
\usepackage{soul}
\usepackage{graphicx}
\usepackage{caption}
\def \cC{\mathcal{C}}
\def \cD{\mathcal{D}}

\def \cS{\mathcal{S}}

\def \cF{\mathcal{F}}
\def \cG{\mathcal{G}}
\def \cI{\mathcal{I}}
\def \cJ{\mathcal{J}}

\def \cO{\mathcal{O}}

\def \cR{\mathcal R}

\def \P{\mathsf P}

\def \E{\mathsf E}

\def \N{\mathbb{N}}
\def \R{\mathbb{R}}

\def \XX{\widetilde{X}}
\def \ud{\mathrm{d}}
\def \e{\mathrm{e}}
\newcommand{\eps}{\varepsilon}
\newcommand{\ind}{\mathbf{1}}

\newtheorem{theorem}{Theorem}[section]
\newtheorem{lemma}[theorem]{Lemma}
\newtheorem{corollary}[theorem]{Corollary}
\newtheorem{proposition}[theorem]{Proposition}
\newtheorem{example}{Example}[section]

\newtheorem{remark}[theorem]{Remark}
\newtheorem{assumption}[theorem]{Assumption}
\title[Optimal stopping problems on finite-time horizon]{Stopping spikes, continuation bays \\ and other features of optimal stopping \\ with finite-time horizon}
\author[T. De Angelis]{Tiziano De Angelis}
\thanks{{\em Acknowledgements}: I am grateful to two anonymous referees whose insightful comments greatly improved the overall quality of the paper.}
\address{T.~De Angelis: School of Management and Economics (Dept.\ ESOMAS) University of Turin, C.so Unione Sovietica 218bis, 10134, Torino (Italy); Collegio Carlo Alberto, Piazza Arbarello 8, 10122, Torino, (Italy).}
\email{\href{mailto:tiziano.deangelis@unito.it}{tiziano.deangelis@unito.it}}
\date{\today}
\numberwithin{equation}{section}
\begin{document}
\begin{abstract}
We consider optimal stopping problems with finite-time horizon and state-dependent discounting. The underlying process is a one-dimensional linear diffusion and the gain function is time-homogeneous and difference of two convex functions. Under mild technical assumptions with local nature we prove fine regularity properties of the optimal stopping boundary including its continuity and strict monotonicity. The latter was never proven with probabilistic arguments. We also show that atoms in the signed measure associated with the second order spatial derivative of the gain function induce geometric properties of the continuation/stopping set that cannot be observed with smoother gain functions (we call them {\em continuation bays} and {\em stopping spikes}). The value function is continuously differentiable in time without any requirement on the smoothness of the gain function.
\end{abstract}

\maketitle

\section{Introduction}

In this paper we analyse in depth some fine properties of optimal stopping problems with finite-time horizon and state-dependent discounting, when the underlying process is a time-homogeneous one-dimensional diffusion and the stopping payoff is also time-homogeneous. Under very mild (local) regularity conditions on the stopping payoff and the diffusion process we provide results concerning the smoothness of the value function (in time) and the geometry of the optimal stopping boundary. 

Denoting by $g$ the stopping payoff (or gain function) and by $X$ the underlying process, we show that when $g$ is just the difference of two convex functions the value function of the problem is continuously differentiable in time. Moreover, the geometry of the stopping set depends in a peculiar way on the interplay between the second-order weak derivative $g''(\ud x)$ (interpreted as a signed measure) and the local-time of $X$, via the so-called {\em Lagrange} formulation of the stopping problem, obtained as an application of It\^o-Tanaka-Meyer formula. Among other things we are able to identify sufficient conditions for the formation of {\em continuation bays} and {\em stopping spikes}, neither of which would occur in the case of a smoother gain function. Both phenomena appear as a result of the presence of atoms in the measure $g''(\ud x)$: continuation bays are associated to positive atoms and stopping spikes are associated to negative ones. It is important to recognise that these features are far from being artificial and indeed occur in very natural optimal stopping problems, as illustrated in Examples \ref{ex:bays} and \ref{ex:spikes} of Section \ref{sec:bays}, including the celebrated American put/call option problem.

One key result of the paper concerns the strict monotonicity of time-dependent optimal stopping boundaries (Corollary \ref{cor:non-flat}). In the literature we can find a wealth of numerical illustrations of optimal boundaries $t\mapsto b(t)$ that exhibit a smooth profile and strict (piecewise) monotonic behaviour (see, e.g., numerous examples in \cite{PeSh}). While the probabilistic study of continuity of the map $t\mapsto b(t)$ has a relatively long history (classical tricks are presented in \cite{PeSh} and more recent results can be found in \cite{DeA15} and \cite{Pe19}) we are not aware of any rigorous probabilistic proof of the strict monotonicity. This question is addressed in Section \ref{sec:cont} of the paper where we give simple sufficient conditions for the strict monotonicity of the optimal boundary and provide a proof based on probabilistic methods and reflecting diffusions. The result complements analogous classical results from the PDE literature, which normally require more stringent conditions on the problem data (traditional references are \cite{Cannon} and \cite{Fr75}, among many others). As an application we show that the optimal exercise boundary of the American put in the classical Black and Scholes model is indeed strictly increasing as a function of time (Example \ref{ex:non-flat-b}). To the best of our knowledge this is the only existing probabilistic proof that does not require any assumption on the smoothness of the boundary itself (PDE methods were used for example in \cite{CC07} and \cite{E04} and we are also aware of a probabilistic proof in \cite{Villeneuve}, which however requires $C^1$-regularity of the optimal boundary).

An important feature of our work, that sets it apart from the majority of papers in the field, is that we conduct a {\em local} study of the problem. That is, we provide our results under assumptions concerning the local behaviour of the underlying diffusion and of the gain function, rather than their global behaviour. This allows wide applicability of our methods to specific problems and extensions beyond our set-up are possible on a case by case basis. As part of our methodology here we study the boundary of the stopping set as a function of the spatial variable, i.e., $x\mapsto c(x)$, rather than as a function of time $t\mapsto b(t)$. This choice is natural, due to the time-homogeneity of both the gain function and the diffusion process, but it is not very common in the literature. Such parametrisation turns out to be very fruitful as we are able to perform a detailed study which includes continuity and strict (piecewise) monotonicity of the map $x\mapsto c(x)$, without requiring any convexity or monotonicity of the gain function (nor any structural assumption on the state-dependent discount rate). Then $x\mapsto c(x)$ can be inverted locally to obtain a local representation of the optimal boundary as a function of time, $t\mapsto b(t)$, which is continuous and strictly monotonic. 

There is a broad class of problems that falls directly within the framework of our paper. Along with the already mentioned American put problem (see, e.g., \cite{J91}), we find numerous other applications from American option pricing (e.g., chooser options \cite{CZ05} and strangle options \cite{DE09}), options embedded in insurance policies (see, e.g., \cite{CDeAS20} and references therein) and technical analysis (see \cite{DeAPe16}). An early contribution to optimal stopping theory fitting in our set-up is \cite{Kotlow}, which proposes a constructive procedure to identify the optimal boundary, based on PDE methods, under the requirement that the gain function be three times continuously differentiable. Stopping problems related to R\"ost's solution of the Skorokhod embedding problem are also covered by the present paper: \cite{MC91} addresses the question from a PDE point of view, \cite{DeA18} from a probabilistic one and \cite{DeAK17} obtains the optimal boundaries numerically (\cite[Rem.\ 3.5]{CW13} and \cite[Rem.\ 17]{CP15} also contain valuable insight on the connection between R\"ost embedding and optimal stopping). Finally, our set-up covers special cases of several theoretical papers on optimal stopping and free boundary problems. Just to mention some early contributions from both the PDE and the probabilistic strand of the literature, we refer to \cite{V99} and \cite{LS09} which consider time-homogeneous gain function and underlying (multidimensional) arithmetic/geometric Brownian motions, \cite{JLL} which also allows for time-inhomogeneous diffusions and, finally, \cite{JL92} which includes both time-inhomogeneous gain function and underlying diffusion.  

It is also worth drawing a parallel with the two closely related papers \cite{LZ13} and \cite{L09} (notice that \cite{L09} mentions \cite{LZ13} as a preprint). In \cite{LZ13}, Lamberton and Zervos study infinite-time horizon optimal stopping problems for one-dimensional linear diffusions when the gain function is time-homogeneous and difference of two convex functions. That paper deals mainly with the variational characterisation of the value function but it also addresses optimal boundaries in specific examples. Differently to the present paper, in \cite{LZ13} the state space is one-dimensional and optimal boundaries are points on the real line. We can think of our paper as an analogue of \cite{LZ13} but in the finite-time horizon setting. The methods used in \cite{LZ13} are those from the theory of one-dimensional diffusions and ordinary differential equations, which do not apply here because the free boundary problems associated to our optimal stopping problems are of parabolic type. The analysis from \cite{LZ13} is then extended by Lamberton in \cite{L09} to the finite-time horizon framework. The underlying process is a one-dimensional linear diffusion and the gain function is time-homogeneous but it is only bounded and Borel measurable as function of the state-process. It is shown in \cite{L09} that the value function of the optimal stopping problem is continuous and it is the unique (bounded continuous) solution of a variational problem understood in the sense of distributions. Our work complements \cite{L09} by focussing on the study of the geometry of the optimal stopping set and on the regularity of its boundary.

The paper is organised as follows. In Section \ref{sec:setup} we formulate the problem and recall useful facts on optimal stopping and one-dimensional linear diffusions. Then we show existence of an optimal boundary $x\mapsto c(x)$ and prove its regularity in the sense of diffusions. Section \ref{sec:C1} is devoted to proving that the time derivative of the value function is continuous in the whole space. Fine geometric properties of the continuation/stopping set are addressed in Section \ref{sec:bays} whereas the continuity of the map $x\mapsto c(x)$ (or equivalently the strict monotonicity of its inverse $t\mapsto b(t)$) is studied in Section \ref{sec:cont}. The paper is completed by a technical appendix.

\section{Setting}\label{sec:setup}

\subsection{The underlying process and the gain function}\label{sec:set1}
Let us consider a complete probability space $(\Omega,\cF,\P)$ equipped with a Brownian motion $B:=(B_t)_{t\ge 0}$ and its filtration $\mathbb{F}:=(\cF_t)_{t\ge 0}$, which is augmented with $\P$-null sets. Let $X:=(X_t)_{t\ge 0}$ be a linear diffusion on an open (possibly unbounded) interval $\cI=(\underline x,\overline x)\subseteq\R$. 
We assume that $X$ be determined as the unique strong solution of the stochastic differential equation (SDE) 
\begin{align}\label{eq:X}
\ud X_t=\sigma(X_t)\ud B_t,\qquad X_0=x,
\end{align}
for a suitable $\sigma:\cI\to [0,\infty)$. We also require that $X$ be a Feller process, hence strong Markov thanks to continuity of paths. We further assume that the diffusion is infinitely-lived, in the sense that the endpoints of the interval $\cI$ are natural in the terminology of \cite[Chapter\ II]{BS} (in particular this means that $\underline x$ and $\overline x$ are not attainable by the process in finite time and the process cannot be started from those points). 

To summarise we impose:
\begin{assumption}\label{ass:SDE}
The process $X$ is strong Markov and it is the unique strong solution of \eqref{eq:X}. The endpoints $\underline x$ and $\overline x$ of $\cI$ are natural. 
\end{assumption}
\noindent To keep the exposition simple, we also make the next mild assumption on the diffusion coefficient.   
\begin{assumption}\label{ass:sigma1}
We have $\sigma\in C(\cI)$ and strictly positive in $\cI$. In particular, for any compact $K\subset \cI$ there exist constants $0<\underline\sigma_K\le \overline\sigma_K<\infty$ such that
\[
\underline\sigma_K\le \sigma(x)\le\overline{\sigma}_K,\quad\text{for all $x\in K$.}
\]
\end{assumption}
Assumptions \ref{ass:SDE} and \ref{ass:sigma1} are enforced throughout the paper.
Sometimes we will use the notation $X^x$ to keep track of the flow property of the process $X$ or, alternatively, we will denote $\P_x(\,\cdot\,):=\P(\,\cdot\,|X_0=x)$.

Fix a time $T\in(0,\infty)$ and continuous functions $r:\cI\to[0,\infty)$ and $g:\cI\to\R$ such that, for any compact $K\subset \cI$, 
\begin{align}\label{eq:integr-g}
\sup_{x\in K}\E_x\left[\sup_{0\le t\le T}\e^{-\int_0^t r(X_s)\ud s}\left|g(X_t)\right|\right]<+\infty.
\end{align}
The problem we are interested in is the following finite-time horizon optimal stopping problem:
\begin{align}\label{eq:value}
v(t,x):=\sup_{0\le\tau\le T-t}\E_x\left[\e^{-\int_0^\tau r(X_s)\ud s}g(X_\tau)\right],\qquad (t,x)\in [0,T]\times\cI,
\end{align}
where the supremum is taken over stopping times of the filtration $\mathbb{F}$. 

Throughout the paper the minimal regularity assumption on the function $g$ is that 
\[
\textit{$g$ can be written as the difference of two convex functions}.
\]
Then, its first (weak) derivative $g'$ exists as a function of bounded variation (which can be taken to be either right or left continuous) and its second (weak) derivative $g''$ exists as a signed measure on $\cI$. For the sake of concreteness, and following \cite[Chapter VI.1]{RY}, we take $g'$ as the {\em left-derivative} of $g$. So $g'$ is left-continuous with right-limits on $\cI$. 
Finally, the measure $g''(\ud x)$ is defined in the usual way via
\[
g''\big([a,b)\big)=g'(b)-g'(a),\quad \underline x\le a<b\le \overline x.
\]

Let us now introduce 
\begin{align}\label{eq:ell}
\ell_t^z=\int_0^t \e^{-\int_0^s r(X_u)\ud u}\ud L^z_s\quad\text{and}\quad \mu(\ud z)=g''(\ud z)-2\sigma^{-2}(z)r(z)g(z)\ud z.
\end{align}
Here $L^z$ denotes the local time of the process $X$ at a point $z\in\cI$, which is defined as
\begin{align}\label{eq:loctime}
L^z_t:=\lim_{\eps \to 0}\frac{1}{2\eps}\int_0^t\ind_{\{X_s\in(z-\eps,z+\eps)\}}\ud \langle X\rangle_s,\quad \P-a.s.,
\end{align}
where $\langle X\rangle_t=\int_0^t\sigma^2(X_s)\ud s$ is the quadratic variation of $X$. For the analysis that follows in the next sections, it is convenient to decompose the signed measure $\mu(\ud x)$ into its positive and negative part (see also Section \ref{sec:bays}):
\[
\mu(\ud x)=\mu^+(\ud x)-\mu^-(\ud x).
\]
Throughout the paper we denote by $\overline A$ the closure of a Borel set $A\subset \R$.

In order to derive the next formula \eqref{eq:u} we need a mild integrability condition:
\begin{assumption}\label{ass:ell}
For every $x\in\cI$
\[
\text{either}\quad \E_x\Big[\int_\cI\ell^z_T\mu^+(\ud z)\Big]<\infty\quad\text{or}\quad \E_x\Big[\int_\cI\ell^z_T\mu^-(\ud z)\Big]<\infty \quad\text{or both}.
\]
\end{assumption}
Thanks to the regularity of $g$ and Assumption \ref{ass:ell} we can use It\^o-Tanaka-Meyer formula (see \cite[Thm.\ VI.1.5]{RY}) to write the problem in the so-called {\em Lagrange} formulation:
\begin{align}\label{eq:u}
u(t,x):=v(t,x)-g(x)=\sup_{0\le\tau\le T-t}\E_x\left[\tfrac{1}{2}\int_{\cI}\ell^z_\tau \mu(\ud z)\right].
\end{align}
Since the endpoints of $\cI$ are natural the expression above follows from \cite[Lemma 3.2]{LZ13}. 
For completeness we provide a full proof in the Appendix.

From standard theory on one-dimensional diffusions it is known that $X$ admits a transition density with respect to the speed measure which is continuous in all its variables (see \cite[Chapter II.1.4]{BS} and \cite[p.\ 149]{IMcK}). In other words, there exists a continuous function
\begin{align}\label{eq:p}
\hat p:(0,\infty)\times\cI\times\cI\to\R_+
\end{align}
such that $\P_x(X_t\in A)=\int_{A}\hat p(t,x,y)m(\ud y)$ for all Borel sets $A\subseteq \cI$ and any $t>0$, where $m(\ud y)$ denotes the speed measure of $X$. In our case, since $\sigma(\,\cdot\,)>0$ on $\cI$, we have 
\begin{align}\label{eq:m}
m(\ud y)= m'(y)\ud y = \frac{2\ud y}{\sigma^{2}(y)S'(y)},\qquad y\in\cI,
\end{align}
where $S'(\,\cdot\,)$ is the derivative of the scale function of the process (since $X$ is in natural scale it simply holds $S'(y)=1$ for $y\in\cI$).
Then $X$ admits a transition density with respect to the Lebesgue measure. 

Before closing this section we make a couple of observations concerning the generality of our model.
\begin{remark}[{\bf Gain function and discount rate}]\label{rem:gain}
The requirement that the discount rate $r$ be non-negative can be easily relaxed to $r:\cI\to[-r_0,\infty)$ for some constant $r_0\in[0,\infty)$. Further relaxations are also possible but at the cost of additional integrability requirements, and we leave such extensions aside for the sake of clarity of exposition.

If $g\in C^2(\cI)$ then the value function $u$ of problem \eqref{eq:u} takes the more familiar form
\begin{align*}
\sup_{0\le\tau\le T-t}\E_x\left[\int^\tau_{0}\e^{-\int_0^s r(X_u)\ud u} \Big(\tfrac{1}{2}\sigma^2(X_s)g''(X_s)-r(X_s)g(X_s)\Big)\ud s\right].
\end{align*}
If we add a running profit function $h:\cI\to\R$, the original problem in \eqref{eq:value} reads
\[
\sup_{0\le \tau \le T}\E_x\left[\int_0^\tau \e^{-\int_0^t r(X_s)\ud s}h(X_t)\ud t+\e^{-\int_0^\tau r(X_s)\ud s}g(X_\tau)\right].
\]
Then the problem in \eqref{eq:u} becomes 
\begin{align*}
\sup_{0\le\tau\le T-t}\E_x\left[\tfrac{1}{2}\int_{\cI}\ell^z_\tau \nu(\ud z)\right],
\end{align*}
with 
\begin{align*}
\nu(\ud z)=g''(\ud z)+2\sigma^{-2}(z)(h(z)-r(z)g(z))\ud z.
\end{align*}
Since all the key results of the paper are based on properties of the measure $\mu(\ud x)$, they immediately carry over to problems in which $\mu(\ud x)$ is replaced by $\nu(\ud x)$ defined above.
\end{remark}

\begin{remark}[{\bf The underlying process}]\label{rem:process}
It is important to notice that there is (almost) no loss of generality in assuming zero drift in the dynamics \eqref{eq:X} of the process $X$. Indeed, let us consider instead a strong solution $Y$ of the SDE 
\begin{align}\label{eq:Y}
\ud Y_t=\alpha(Y_t)\ud t+\beta(Y_t)\ud B_t,\qquad Y_0=y,
\end{align}
on some interval $\hat \cI\subseteq\R$, with $\alpha$ and $\beta$ drift and diffusion coefficients that guarantee existence and uniqueness of the strong solution. Assume the endpoints of the interval $\hat \cI$ are natural. Let us then consider the stopping problem
\[
\sup_{0\le\tau\le T}\E_x\left[\e^{-\int_0^\tau \hat r(Y_t)\ud t}\hat g(Y_\tau)\right],
\]
with Borel measurable functions $\hat g:\hat \cI\to\R$ and $\hat r:\hat \cI\to [0,\infty)$. Then we can reduce to the setting of \eqref{eq:X} and \eqref{eq:value} by a simple change of scale. That is, letting $\hat S$ be the scale function of $Y$, we have that $X:=\hat S(Y)$ solves \eqref{eq:X} with $\sigma(x)=(\hat S'\circ \hat S^{-1})(x)(\beta\circ \hat S^{-1})(x)$, and the stopping problem takes the form of \eqref{eq:value} with $g(x)=(\hat g\circ \hat S^{-1})(x)$ and $r(x)=(\hat r\circ \hat S^{-1})(x)$. Obviously here $\cI=\hat S (\hat \cI)$.

This approach can be extended even further to consider SDEs with generalised drift (in the sense of, e.g., \cite{ZR19}), again by adopting the change of coordinates via the scale function. However, we insist on the requirement that the SDE admits a unique strong solution because we use pathwise uniqueness in our arguments below (recall that weak existence and pathwise uniqueness imply strong existence; see, e.g., \cite[Corollary 5.3.23]{KS}).  
\end{remark}

\subsection{Generalities on the value function and existence of a boundary}
Since $X$ admits a continuous transition density with respect to its speed measure, then the two-dimensional process $(t,X)$ enjoys Feller property. The latter, combined with continuity of $g$ and with \eqref{eq:integr-g}, is known to be sufficient to obtain that
\[
(t,x)\mapsto v(t,x)\quad\text{is lower semi-continuous}
\]
thanks to standard results from \cite[Lemma 3, Sec.\ 3.2.3 and Lemma 4, Sec.\ 3.2.4]{Sh}. 
As a consequence of \eqref{eq:integr-g} we also have
\begin{align}\label{eq:bound-v}
\sup_{(t,x)\in [0,T]\times K}|v(t,x)|<\infty, \quad\text{for any compact $K\subset\cI$}.
\end{align}

Further, we know from \cite[Thm.\ 1, Sec.\ 3.3.1 and Thm.\ 3, Sec.\ 3.3.3]{Sh} that  
\[
\tau_*=\inf\{t\in[0,T]: (t,X_t)\in\cS\}
\]
is the minimal optimal stopping time, where
\[
\cS=\{(t,x)\in[0,T]\times\cI\,:\,v(t,x)=g(x)\}
\]
is the so-called {\em stopping set} (notice in particular that $\{T\}\times\cI\subseteq\cS$ by definition). We will sometimes use the notation $\tau^{t,x}_*$ to emphasise the dependence of this stopping time on the initial position of the time-space process; that is
\[
\tau^{t,x}_*=\inf\{s\in[0,T-t]: (t+s,X^x_s)\in\cS\}.
\]
Letting also the {\em continuation set} be denoted by
\[
\cC=\{(t,x)\in[0,T]\times\cI\,:\,v(t,x)>g(x)\},
\]
we immediately see that $\cC$ is open and $\cS$ is closed (relative to $[0,T]\times\cI$) thanks to lower semi-continuity of $u=v-g$. Given an interval $\cJ\subseteq\cI$, it will be sometimes convenient to work with sets of the form 
\[
\cC_{\cJ}=\cC\cap\big([0,T)\times\mathrm{int}(\cJ)\big)
\]
and the associated boundary $\partial\cC_{\cJ}=\partial\cC\cap\big([0,T)\times\mathrm{int}(\cJ)\big)$. We should always understand $\partial \cC$ as $\partial \cC_{\cI}$.

Finally, it follows from standard theory (\cite[Sec.\ 3.4]{Sh}) that for $s\in[0,T-t]$ 
\begin{align}\label{eq:martv}
s\mapsto \e^{-\int_0^s r(X_u)\ud u}v(t+s,X_{s})\quad\text{is a super-martingale}
\end{align}
and, equivalently, from \eqref{eq:u}
\begin{align}\label{eq:mart}
s\mapsto \e^{-\int_0^s r(X_u)\ud u} u(t+s,X_{s})+\tfrac{1}{2}\int_\cI \ell^z_{s}\mu(\ud z)\quad\text{is a super-martingale}.
\end{align}
If we replace $s$ with $s\wedge\tau_*$ the two stopped processes above are martingales on $[0,T-t]$.

In our setting, the process $X$ is time-homogeneous and the gain function $g$ is independent of time. Therefore, it is immediate to verify that for any $(t,x)\in(0,T]\times\cI$ and $h\in(0, t)$
\begin{align}\label{eq:v-mon}
v(t,x)=&\sup_{0\le\tau\le T-t}\E_x\left[\e^{-\int_0^\tau r(X_s)\ud s}g(X_\tau)\right]\\
\le& \sup_{0\le\tau\le T-t+h}\E_x\left[\e^{-\int_0^\tau r(X_s)\ud s}g(X_\tau)\right]=v(t-h,x),\notag
\end{align}
so that 
\begin{align}\label{eq:monot-u}
t\mapsto u(t,x)=v(t,x)-g(x)\quad\text{is non-increasing}.
\end{align}
Such monotonicity of the value function identifies some sort of `privileged' direction in the state space, in the sense of the following simple statement
\[
(t,x)\in\cS\implies[t,T]\times\{x\}\subseteq\cS.
\]
Then, we can uniquely determine the boundary of the continuation set by defining
\[
c(x):=\sup\{t\in[0,T): v(t,x)>g(x)\}\wedge T\quad\text{with $\sup\varnothing=0$}.
\]

Since $\cS$ is a closed set we have that for any $x\in\cI$ and any sequence $x_n\to x$, as $n\to \infty$, it holds
\[
\cS\ni\liminf_{n\to \infty}(c(x_n),x_n)=(\liminf_{n\to\infty}c(x_n),x),
\]
hence
\[
\liminf_{n\to \infty}c(x_n)\ge c(x).
\]

In conclusion we can summarise the above discussion in the next proposition.
\begin{proposition}\label{prop:boundary1}
The stopping set can be expressed as
\[
\cS=\{(t,x)\in[0,T]\times\cI\,:\,t\ge c(x)\},
\] 
where $c:\cI\to[0,T]$ is lower semi-continuous.
\end{proposition}

\begin{remark}
It is worth noticing that, in the literature on finite-time horizon optimal stopping problems on a one-dimensional diffusion, it is customary to describe the stopping set in terms of a boundary which is time-dependent, rather than space-dependent. However, proving existence of such boundaries requires to show that, e.g., the map $x\mapsto u(t,x)$ is monotonic, or at least convex. This type of argument will fail in general, if $g$ is just the difference of two convex functions and in the presence of a state-dependent discount rate.

Instead, the existence of the boundary $x\mapsto c(x)$ is an immediate consequence of the monotonicity in time of the value function, which holds under even wider generality than ours. Indeed, a quick look at the argument in \eqref{eq:v-mon} reveals that monotonicity of $t\mapsto v(t,x)$ is solely a consequence of the reduction of the set of admissible stopping times and has nothing to do with either the process $X$, the discount rate or the gain function (provided the latter three are time-homogeneous).   
\end{remark}

For some of the results that follow, it is convenient to work with a continuous value function $v$. Continuity is normally easy to prove in specific examples but also general results exist. For example, if $g$ is bounded and continuous and $r(x)\equiv r\ge 0$ \cite[Thm.~4.3]{PS11} guarantees continuity of $v$ when $X$ is just a Feller process on a locally compact and separable space (not necessarily a solution of an SDE). Alternatively, for $X$ a non-exploding linear diffusion on an interval and $r(x)\ge0$ locally bounded and measurable, \cite{L09} proves continuity of $v$ when $g$ is just bounded and measurable. Rather than giving another proof of the continuity of the value function, when necessary we will invoke the next assumption.
\begin{assumption}\label{ass:cont-v}
We have $v\in C([0,T]\times\cI)$.
\end{assumption}
For $\alpha\in(0,1)$ we denote by $C^\alpha_{\ell oc}(\cI)$ the class of locally $\alpha$-H\"older continuous functions on $\cI$. 
Continuity of the value function, along with the martingale property, Assumption \ref{ass:sigma1} and a standard PDE argument gives the next corollary (see, e.g., the proof of \cite[Prop.~7.7, Ch.~2]{KS2}).
\begin{corollary}\label{cor:PDE}
Under Assumption \ref{ass:cont-v}, if $r,\sigma\in C^\alpha_{\ell oc}(\cI)$ for some $\alpha\in(0,1)$ we have $v\in C^{1,2}(\cC)$ and it solves
\[
\partial_tv(t,x)+\tfrac{1}{2}\sigma^2(x)\partial_{xx}v(t,x)=r(x)v(t,x),\qquad (t,x)\in\cC,
\]
with boundary condition $v=g$ on $\partial\cC$.
\end{corollary}
\begin{remark}\label{rem:loc-reg}
The argument of proof of the corollary is based on exit times from open bounded sets contained in $\cC$. As such, it has a `{\em local}' nature that allows to relax the continuity assumption on the diffusion coefficient and discount rate. For our purposes, in Theorem \ref{thm:non-flat}, we will consider the case in which $r,\sigma\in C^\alpha_{\ell oc}(\cJ)$ with $\cJ\subseteq \cI$ an open set. Then $v\in C^{1,2}\big(\cC_{\cJ}\big)$ and it satisfies 
\begin{align}\label{eq:PDEv}
\partial_tv(t,x)+\tfrac{1}{2}\sigma^2(x)\partial_{xx}v(t,x)=r(x)v(t,x),\qquad (t,x)\in\cC_{\cJ}.
\end{align}
\end{remark}

\subsection{Regularity of the boundary in the sense of diffusions} 
There is an important consequence to Proposition \ref{prop:boundary1}. Indeed, it turns out that the boundary of the continuation set, $\partial\cC$, is {\em regular} for the stopping set in the sense of diffusions. We will review this property in detail below.

Denote the hitting time to the stopping set by
\[
\sigma_*=\inf\{t\in(0,T]: (t,X_t)\in\cS\}.
\]
As before we denote $\sigma^{t,x}_*$ when we need to emphasise the initial point of the process $(t,X)$: that is
\[
\sigma^{t,x}_*=\inf\{s\in(0,T-t]: (t+s,X^x_s)\in\cS\}.
\]
For any $(t,x)\in[0,T]\times\cI$, it is clear that $\tau^{t,x}_*\le \sigma^{t,x}_*$, $\P$-a.s., by definition. By continuity of paths $t\mapsto X_t$ and the fact that $\cC$ and $\mathrm{int}(\cS)$ are open sets (provided they are not empty), it is also clear that $\tau^{t,x}_*= \sigma^{t,x}_*$, $\P$-a.s., for all $(t,x)\in\cC\cup\mathrm{int}(\cS)$. In principle the strict inequality $\tau^{t,x}_*<\sigma^{t,x}_*$ might hold with positive probability at points on the boundary $\partial\cC=\partial\cC_\cI$. We are going to prove below that this does not happen in our setting. For that we need the following simple lemma, which we believe to be well-known but whose proof may be hard to locate in the literature; we give a proof in the Appendix for completeness.
\begin{lemma}\label{lem:contBM}
Consider a Brownian motion $W:=(W_t)_{t\ge 0}$, an interval $(a,b)$ and the stopping time
\[
\tau^x_{(a,b)}=\inf\{t\ge 0:x+W_t\notin (a,b)\},\quad x\in[a,b].
\]
For any $x\in[a,b]$ and any sequence $(x_n)_{n\ge 1}\subset(a,b)$ such that $x_n\to x$ as $n\to\infty$, we have 
\[ 
\lim_{n\to\infty}\tau^{x_n}_{(a,b)}=\tau^x_{(a,b)},\quad\text{$\P$-a.s.}
\] 
\end{lemma}
It is worth noticing that $\tau^x_{(a,b)}=\tau^x_a\wedge\tau^x_b$ with $\tau^x_c=\inf\{t\ge 0: x+W_t=c\}$ for $c\in\R$. Then we have $\tau^x_c=\tau^0_{c-x}$ and Lemma \ref{lem:contBM} holds as soon as we show that $\tau^0_{c_n}\to\tau^0_{c}$, $\P$-a.s., for any $c\in\R$ and any sequence $c_n\to c$. It is well-known that $(\tau^0_c)_{c\ge 0}$ is an increasing L\`evy process and it is $\P$-a.s.\ left-continuous \cite[Prop.\ III.3.8]{RY}. Moreover, there is almost surely no {\em interval} over which it is continuous. However, for a fixed $c\ge 0$ and any sequence $c_n\downarrow c$ we have $\tau^0_{c_n}\downarrow\tau^0_c$, $\P$-a.s.\ \cite[Prop.\ III.3.9]{RY}. Combining the latter with left-continuity we have that $\tau^0_{c_n}\to\tau^0_{c}$, $\P$-a.s., for any $c\ge 0$ and any sequence $c_n\to c$. Noticing that $\tau^0_c=\inf\{t\ge 0:\sup_{0\le u\le t}W_s=c\}$ for $c\ge 0$, it is immediate to extend the arguments to the family $(\tau^0_c)_{c\in \R}$ by considering also $\tau^0_c=\inf\{t\ge 0:\inf_{0\le u\le t}W_s=c\}$ for $c\le 0$. Thus, the result in Lemma \ref{lem:contBM} can be deduced also from these standard facts\footnote{I am grateful to G.\ Peskir for showing me this elegant argument.}.

Building upon the previous lemma we can prove our next result.
\begin{proposition}\label{prop:reg-b}
Let $(t_0,x_0)\in\partial\cC$ and let $(t_n,x_n)_{n\ge 1}\subset\cC$ be a sequence that converges to $(t_0,x_0)$. Then,
\begin{align}\label{eq:sigma-cont}
\lim_{n\to \infty}\P(\sigma_*^{t_n,x_n}\ge \eps)=0,\quad\text{for any $\eps>0$}.
\end{align}
\end{proposition}
\begin{proof}
We give the proof in two steps. For simplicity and with no loss of generality we assume that $t\mapsto B_t(\omega)$ is continuous and it satisfies the law of iterated logarithm for all $\omega\in\Omega$.
\vspace{+3pt}

{\bf Step 1}. ({\em $\sigma\equiv1$}). First we prove the result for $\sigma(x)\equiv 1$, so that $X$ is a standard Brownian motion, $\cI=\R$, and the main ideas in the proof are more transparent. Actually, here we prove a stronger result and show that 
\begin{align}\label{eq:sigma-limsup}
\limsup_{n\to \infty}\sigma_*^{t_n,x_n}(\omega)=0,\quad\text{for all $\omega\in \Omega$.}
\end{align}

If $(t_0,x_0)\in\partial\cC$, then by definition $t_0\ge c(x_0)$ and $[t_0,T]\times\{x_0\}\subset\cS$. By the law of iterated logarithm we know that for any $\omega\in\Omega$ and any $\eps>0$, there exist $0<s^\eps_1(\omega)<s^\eps_2(\omega)<\eps$ such that
\[
x_0+B_{s^\eps_1}(\omega)<x_0<x_0+B_{s^\eps_2}(\omega).
\]
Hence, $\sigma_*^{t_0,x_0}(\omega)=0$, for all $\omega\in\Omega$. Let us now fix $\eps>0$ and $\omega\in\Omega$. Then there is $\delta_\omega>0$ such that 
\[
x_0+B_{s^\eps_1}(\omega)<x_0-\delta_\omega\quad\text{and}\quad x_0+\delta_\omega<x_0+B_{s^\eps_2}(\omega).
\]
Therefore, taking $n$ sufficiently large we have $t_n+s^\eps_1(\omega)\ge t_0$ and $|x_n-x_0|<\delta_\omega/2$, so that 
\[
x_n+B_{s^\eps_1}(\omega)<x_0<x_n+B_{s^\eps_2}(\omega)\quad\text{and}\quad t_n+s^\eps_1(\omega)\ge t_0.
\]
Hence, there must exist $\hat s_n(\omega)\in\big(s^\eps_1(\omega),s^\eps_2(\omega)\big)$ such that $(t_n+\hat s_n,x_n+B_{\hat s_n}(\omega))\in[t_0,T]\times\{x_0\}$, which implies 
\[
\limsup_{n\to\infty}\sigma^{t_n,x_n}_*(\omega)\le \eps.
\]
Since $\eps>0$ was arbitrary we have \eqref{eq:sigma-limsup}.
\vspace{+3pt}

{\bf Step 2}. ({\em Any $\sigma>0$}). Let us now consider a generic diffusion coefficient $\sigma$ that satisfies Assumption \ref{ass:sigma1}. The rest of this proof is slightly technical due to the fact that we localise the dynamics of $X$ on a bounded open interval $\cJ$ with $\overline\cJ \subset\cI$. 

Denote
\[
\tau_\cJ=\inf\{t\ge 0: X_t\notin\cJ\}
\]
and
\begin{align}\label{eq:mJ}
M_{t\wedge\tau_\cJ}\!=\!\int_0^{t\wedge\tau_\cJ}\!\!\sigma(X_s)\ud B_s\,,\quad\! \langle M\rangle_{t\wedge\tau_\cJ}\!=\!\int_0^{t\wedge\tau_\cJ}\!\!\sigma^2(X_s)\ud s\quad\!\text{and}\quad\! m_\cJ=\langle M\rangle_{\tau_\cJ}.
\end{align}
The process $(M_{t\wedge\tau_\cJ})_{t\ge0}$ is absorbed when $X$ leaves the interval $\cJ$, it is a martingale thanks to Assumption \ref{ass:sigma1}, and can be represented, by Dambis-Dubins-Schwarz theorem (\cite[Thm.\ IV.34.11]{RW}; see also \cite[Thm.\ V.47.1]{RW}), as a time-changed Brownian motion. That is $M_{t\wedge\tau_\cJ}=W_{\langle M\rangle_{t\wedge\tau_\cJ}}$, where $W$ is a standard Brownian motion. Notice that $(\tau_\cJ, M, \langle M\rangle, m_\cJ)$ depend on the initial point $x\in\cJ$ but for now we omit the dependence for simplicity. 

Since $\sigma(\,\cdot\,)\ge \underline{\sigma}_\cJ>0$ on $\overline\cJ$ by Assumption \ref{ass:sigma1}, we can define the inverse of the quadratic variation $A_s:=(\langle M\rangle)^{-1}(s)$ for $s\in[0,m_\cJ)$ (this is done $\omega$ by $\omega$). Thanks to strict monotonicity of both $A$ and $\langle M\rangle$, both processes are also continuous and clearly $\langle M\rangle_{A_s\wedge\tau_\cJ}=s\wedge m_\cJ$. Then, we have
\begin{align}\label{eq:XZ}
X^x_{A_s\wedge\tau_\cJ}=x+W_{\langle M\rangle_{A_s\wedge\tau_\cJ}}=x+W_{s\wedge m_\cJ}=:Z^x_{s\wedge m_\cJ}.
\end{align}
Here, the process $Z$ depends on $x\in\cJ$, via the initial condition $Z^x_0=x$, the stopping time $m_\cJ=m^x_\cJ$ and the Brownian motion $W=W^x$ obtained via time-change.
By construction we have that $A_s<\tau_\cJ\iff s<m_\cJ$, $\P$-a.s., and $X^x_{A_s}=Z^x_{s}\in \cJ$ for $s<m_\cJ$, $\P$-a.s. Moreover 
\[
Z^x_{m_\cJ}=X^{x}_{\tau_\cJ}\in \partial\cJ\quad\text{on $\{\tau_\cJ<\infty\}$,}
\]
so that, in particular, 
\begin{align}\label{eq:mJ2}
m^x_\cJ=\inf\{s\ge0: Z^x_s\notin\cJ\}.
\end{align}
In conclusion $s\mapsto Z^x_{s\wedge m_\cJ}$ is a Brownian motion absorbed upon leaving the interval $\cJ$ and it is adapted to the time-changed filtration $(\cG_s)_{s\ge 0}=(\cF_{A_s})_{s\ge 0}$.

Having defined $Z$ we can write $A_s=(\langle M\rangle)^{-1}(s)$ explicitly ($\omega$ by $\omega$) as
\begin{align}\label{eq:A}
A_s=\int_0^s\sigma^{-2}(Z_u)\ud u,\quad s\in[0,m_\cJ).
\end{align}
Again, we notice that $A_s=A^x_s$ depends on $x\in\cJ$. Since we are interested in the event $\sigma^{t,x}_*=0$ and we want to restrict our attention to the behaviour of the process $X$ (equivalently $Z$) for `small times', here we will always consider $\sigma^{t,x}_*\wedge\tau_\cJ^x$.
In particular, using that $u\mapsto A^x_u$ is {\em strictly increasing}, we can write 
\begin{align*}
\sigma^{t,x}_*\wedge\tau^x_\cJ\le &\inf\{s>0:(t+s,X^x_s)\in\cS\}\wedge\tau^x_\cJ\\
=&\inf\{A^x_u>0:(t+A^x_u,Z^x_u)\in\cS\}\wedge\tau^x_\cJ,
\end{align*}
where the inequality is due to replacing $s\in(0,T]$ with $s>0$ in the definition of $\sigma_*$, and the final expression by simply relabelling the time variable $s=A^x_u$. Then, setting 
\[
\zeta^{t,x}_*=\inf\{u>0:(t+A^x_u,x+W^x_u)\in\cS\},
\]
we obtain $\sigma^{t,x}_*\wedge\tau^x_\cJ\le A^x_\cdot\circ(\zeta^{t,x}_*\wedge m^x_\cJ)$. By construction we also have $\tau^x_\cJ=A^x_{m^x_\cJ}$. Using these inequalities we obtain 
\begin{align}\label{eq:PP}
\P(\sigma^{t,x}_*\ge \eps)=&\,\P(\sigma^{t,x}_*\ge \eps,\sigma^{t,x}_*<\tau^x_\cJ)+\P(\sigma^{t,x}_*\ge \eps,\sigma^{t,x}_*\ge\tau^x_\cJ)\\
\le&\,\P(\sigma^{t,x}_*\wedge\tau^x_\cJ\ge \eps)+\P(\sigma^{t,x}_*\ge\tau^x_\cJ)\notag\\
\le&\P(\sigma^{t,x}_*\wedge\tau^x_\cJ\ge \eps)+\P(\sigma^{t,x}_*\wedge\tau^x_\cJ>\delta)+\P(\tau^{x}_\cJ\le\delta)\notag\\
\le&\P\Big(A^x_\cdot\circ(\zeta^{t,x}_*\wedge m^x_\cJ)\ge \eps\Big)+\P\Big(A^x_\cdot\circ (\zeta^{t,x}_*\wedge m^x_\cJ)>\delta\Big)\notag\\
&\,+\P\Big(A^x_{m^x_\cJ}\le\delta\Big),\notag
\end{align}
for any $\delta>0$ given and fixed. So, it is sufficient to prove that the final expression above converges to zero along any sequence $\cC_\cJ\supset (t_n,x_n)\to(t_0,x_0)\in\partial\cC_\cJ$. 

It is not very convenient to work with the Brownian motion $W^x$ when $x$ varies. We therefore set
\begin{align*}
&\widetilde{Z}^x_t=x+B_t,\qquad\qquad \widetilde{A}^x_t=\int_0^t\sigma^{-2}(\widetilde{Z}^x_s)\ud s,\\
&\widetilde{\zeta}^{t,x}_*=\inf\{u>0:(t+\widetilde{A}^x_u,\widetilde{Z}^x_u)\in\cS\}\quad\text{and}\quad\widetilde{m}^x_\cJ=\inf\{s\ge0: \widetilde{Z}^x_s\notin\cJ\},
\end{align*}
with $B$ our original Brownian motion.
Then, for each $(t,x)$ we have the equivalence in law
\begin{align}\label{eq:law}
\mathsf{Law}_{\P}(Z^x,A^x,\zeta^{t,x}_*,m_{\cJ}^x)=\mathsf{Law}_{\P}(\widetilde{Z}^x,\widetilde{A}^x,\widetilde{\zeta}^{t,x}_*,\widetilde{m}_{\cJ}^x ).
\end{align}
From \eqref{eq:PP} and the equality in law we have
\begin{align}\label{eq:PP2}
\P(\sigma^{t,x}_*\ge \eps)\le&\, \P\Big(\widetilde{A}^x_\cdot\circ(\widetilde{\zeta}^{t,x}_*\wedge \widetilde{m}^x_\cJ)\ge \eps\Big)+\P\Big(\widetilde{A}^x_\cdot\circ (\widetilde{\zeta}^{t,x}_*\wedge\widetilde{m}^x_{\cJ})>\delta\Big)\\
&\,+\P\Big(\widetilde{A}^x_{\widetilde{m}^x_\cJ}\le\delta\Big),\notag
\end{align}
for any $\delta>0$.

The advantage of working with $(\widetilde{Z}^x,\widetilde{A}^x,\widetilde{\zeta}^{t,x}_*,\widetilde{m}_{\cJ}^x)$ is that $\widetilde{Z}^x$ only depends on $x$ via its initial point and therefore we can apply arguments analogous to those used in step 1 above. An initial observation is that 
\begin{align}\label{eq:AA}
\frac{1}{\underline{\sigma}_\cJ^2}(t\wedge \widetilde{m}^x_\cJ)\ge \widetilde{A}^x_{t\wedge \widetilde{m}^x_\cJ}\ge \frac{1}{\overline{\sigma}_\cJ^2}(t\wedge \widetilde{m}^x_\cJ),\quad\P-a.s.
\end{align}
thanks to Assumption \ref{ass:sigma1} and by definition of $\widetilde{A}$. Take $(t_0,x_0)\in\partial\cC_\cJ$, so that $[t_0,T]\times\{x_0\}\subset\cS$. By the exact same argument as in step 1 above we obtain $\widetilde{\zeta}^{t_0,x_0}_*=0$, $\P$-a.s. Then, \eqref{eq:AA} 
implies $\widetilde{A}^{x_0}_\cdot\circ(\widetilde{\zeta}^{t_0,x_0}_*\wedge \widetilde{m}^{x_0}_\cJ)=0$, $\P$-a.s.\ as well. Next we are going to prove that
\begin{align}\label{eq:cont-zeta}
\limsup_{n\to\infty}(\widetilde{\zeta}^{t_n,x_n}_*\wedge\widetilde{m}^{x_n}_\cJ)=0,\quad\P-a.s.,
\end{align}
so that in \eqref{eq:PP2} we have
\begin{align}\label{eq:PP3}
\P\Big(\widetilde{A}^{x_n}_\cdot\circ(\widetilde{\zeta}^{t_n,x_n}_*\wedge \widetilde{m}^{x_n}_\cJ)\ge \eps\Big)+\P\Big(\widetilde{A}^{x_n}_\cdot\circ (\widetilde{\zeta}^{t_n,x_n}_*\wedge\widetilde{m}^{x_n}_\cJ)>\delta\Big)\to 0
\end{align}
thanks to \eqref{eq:AA}.

Since $\widetilde m_\cJ$ is the first exit time of the Brownian motion $\widetilde Z$ from an open interval, then $x\mapsto \widetilde m^x_\cJ(\omega)$ is continuous for any $\omega \in \Omega$ (in the sense of Lemma \ref{lem:contBM}). Fix $\omega\in\Omega$. By continuity of paths $\widetilde m^{x_0}_\cJ(\omega)>0$ and, in particular, there exists $\eps_{0,\omega}>0$ such that $\widetilde m^{x_0}_\cJ(\omega)\ge \eps_{0,\omega}$. By continuity of $x\mapsto \widetilde m^x_\cJ(\omega)$ we can assume that $\widetilde m^{x_n}_\cJ(\omega)\ge \eps_{0,\omega}/2$ for all $n\ge N_{\omega}$, with $N_{\omega}$ sufficiently large.

Now, we can repeat the arguments from step 1. For any $\eps\in(0,\eps_{0,\omega}/2)$, there exist $0<s^\eps_1(\omega)<s^\eps_2(\omega)<\eps$ and $\delta_\omega>0$ such that
\[
x_0+B_{s^\eps_1}(\omega)<x_0-\delta_\omega\quad\text{and}\quad x_0+\delta_\omega<x_0+B_{s^\eps_2}(\omega).
\]
Clearly, for $n\ge N_\omega$ sufficiently large we have $|x_n-x_0|<\delta_\omega/2$. 
Given that 
$\widetilde{m}^{x_n}_\cJ(\omega)\ge \eps_{0,\omega}/2$ and $\widetilde{A}^{x_n}_{s^\eps_1}(\omega)\ge (1/\overline{\sigma}^2_\cJ)s^\eps_1(\omega)$ by \eqref{eq:AA}, we can also pick $n\ge N_\omega$ sufficiently large such that $t_n+\widetilde{A}^{x_n}_{s_1^\eps}(\omega)\ge t_0$. 

Therefore, taking $n$ sufficiently large we have 
\[
x_n+B_{s^\eps_1}(\omega)<x_0<x_n+B_{s^\eps_2}(\omega)\quad\text{and}\quad t_n+\widetilde{A}^{x_n}_{s_1^\eps}(\omega)\ge t_0,
\]
which implies
\[
\limsup_{n\to\infty}(\widetilde{\zeta}^{t_n,x_n}_*(\omega)\wedge\widetilde{m}^{x_n}_\cJ(\omega))\le \eps.
\]
The argument holds for any $\eps\in(0,\eps_{0,\omega}/2)$. Hence, \eqref{eq:cont-zeta} holds by arbitrariness of $\omega\in\Omega$, and \eqref{eq:PP3} follows. From \eqref{eq:PP2}, \eqref{eq:AA} and \eqref{eq:PP3} we have
\begin{align*}
\limsup_{n\to\infty}\P(\sigma^{t_n,x_n}_*\ge \eps)\le&\, \limsup_{n\to\infty}\P\Big(\widetilde{A}^{x_n}_{\widetilde{m}^{x_n}_\cJ}\le\delta\Big)\\
\le &\,\limsup_{n\to\infty}\P\Big(\widetilde{m}^{x_n}_\cJ\le\overline{\sigma}^2_\cJ\delta\Big)\le\P\Big(\widetilde{m}^{x_0}_\cJ\le\overline{\sigma}^2_\cJ\delta\Big), 
\end{align*}
where the final inequality is by Fubini's theorem and using $\widetilde{m}^{x_n}_{\cJ}\to\widetilde{m}^{x_0}_{\cJ}$, $\P$-a.s.\ by Lemma \ref{lem:contBM}. Since $\P(\widetilde{m}^{x_0}_{\cJ}>0)=1$, letting $\delta\downarrow 0$ we arrive at \eqref{eq:sigma-cont}.
\end{proof}

\section{Regularity of the value function}\label{sec:C1}

In this section we show that the value function has a modulus of continuity with respect to the time variable and, under mild additional assumptions, it is indeed a locally Lipschitz function of time. Our proof uses properties of the local time of the process (generalising \cite[Example 17]{DeAPe}). 
For that we recall that the scale function density is $S'(x)=1$ and that $\hat p$ is the transition density of the process with respect to the speed measure (Section \ref{sec:set1}). 
First we state an estimate for the local time of the process.
\begin{lemma}\label{lem:loc-time}
Let $0< t_1\le t_2\le T$ and fix $x\in\cI$. Then, for any $z\in\cI$ we have  
\begin{align}\label{eq:L0}
\E_x\left[\ell^{z}_{t_2}-\ell^{z}_{t_1}\right]\le \int_{t_1}^{t_2} 2 \hat p(s,x,z)\ud s.
\end{align}
\end{lemma}
\begin{proof}
Thanks to \eqref{eq:loctime} we can select a sequence $(\eps_n)_{n\ge 1}$ such that $\eps_n\to 0$ as $n\to \infty$ and
\begin{align*}
L^{z}_{t_2}-L^{z}_{t_1}=\lim_{n\to\infty}\frac{1}{2\eps_n}\int_{t_1}^{t_2} \ind_{\{|X_s-z|\le \eps_n\}}\sigma^2(X_s)\ud s,\qquad\P_x-a.s.
\end{align*}
By Fatou's lemma and the definition of $\ell^z$ in \eqref{eq:ell} we get
\begin{align}\label{eq:L1}
\E_x\left[\ell^{z}_{t_2}-\ell^{z}_{t_1}\right]&=\E_x\left[\int_{t_1}^{t_2}\e^{-\int_0^s r(X_u)\ud u}\ud L^z_s\right]\le\E_x\left[L^{z}_{t_2}-L^{z}_{t_1}\right]\\
&\le\liminf_{n\to\infty}\frac{1}{2\eps_n}\int_{t_1}^{t_2} \E_x\left[\ind_{\{|X_s-z|\le \eps_n\}}\sigma^2(X_s)\right]\ud s,\notag
\end{align}
where we used $r\ge 0$ for the first inequality.
Writing the expectation in terms of the transition density $\hat p$ and the speed measure (see \eqref{eq:p} and \eqref{eq:m}) we obtain
\begin{align}\label{eq:cz}
\E_x\left[\ind_{\{|X_s-z|\le \eps_n\}}\sigma^2(X_s)\right]=\int^{z+\eps_n}_{z-\eps_n}2\hat p(s,x,y)\ud y\le c(t_1,x,z) 2\eps_n,
\end{align}
where
\[
c(t_1,x,z):=\sup\left\{2\hat p(s,x,y),\,(s,y)\in[t_1,T]\times [z-\eps_0,z+\eps_0]\right\},
\]
for some $\eps_0\ge \eps_n$, $n\ge 1$, upon recalling that $S'(y)=1$. Notice that we are using continuity of $\hat p$ on $(0,\infty)\times\cI\times\cI$.

Thanks to \eqref{eq:cz} we can invoke dominated convergence to pass to the limit in \eqref{eq:L1} and obtain 
\begin{align*}
\E_x\left[\ell^{z}_{t_2}-\ell^{z}_{t_1}\right]\le&\int_{t_1}^{t_2} \limsup_{n\to\infty}\frac{1}{2\eps_n}\E_x\left[\ind_{\{|X_s-z|\le \eps_n\}}\sigma^2(X_s)\right]\ud s=\int_{t_1}^{t_2} 2 \hat p(s,x,z)\ud s.\notag
\end{align*}
\end{proof}

Next we obtain a modulus of continuity for the value function with respect to time. Recall the decomposition of the signed measure $\mu=\mu^+-\mu^-$ into its positive and negative part.
\begin{proposition}\label{prop:lip-t}
For any $x\in\cI$ and any $t_1<t_2$ in $[0,T)$ we have
\begin{align}\label{eq:lip-v}
0\le v(t_1,x)-v(t_2,x)\le \int_{T-t_2}^{T-t_1}\int_\cI 2 \hat p(s,x,z) \mu^+(\ud z)\ud s. 
\end{align}
In particular, if there exists a constant $\kappa=\kappa(t_2,x)>0$, depending on $t_2$ and $x$, such that
\begin{align}\label{eq:boundTD}
\sup_{s\in[T-t_2,T-t_1]}\int_\cI 2 \hat p(s,x,z) \mu^{+}(\ud z)\le \kappa,
\end{align}
then $t\mapsto v(t,x)$ is locally Lipschitz. 
\end{proposition}
\begin{proof}
Clearly $v(t_1,x)\ge v(t_2,x)$ by monotonicity of $v(\,\cdot\,,x)$. For the remaining inequality we use the representation of the problem in terms of the function $u=v-g$ (see \eqref{eq:u}). Let $\tau_*^1=\tau^{t_1,x}_*$ denote the optimal stopping time for the problem started at $(t_1,x)$. Then $\tau_2=\tau^1_*\wedge(T-t_2)$ is admissible and sub-optimal for the problem started at $(t_2,x)$. This gives
\begin{align*}
v(t_1,x)-v(t_2,x)\le& \int_\cI\E_x\left[\ell^z_{\tau^1_*}-\ell^z_{\tau_2}\right] \mu(\ud z)\\
=&\int_\cI\E_x\left[\left(\ell^z_{\tau^1_*}-\ell^z_{T-t_2}\right)\ind_{\{\tau^1_*>T-t_2\}}\right] \mu(\ud z)\\
\le &\int_\cI\E_x\left[\left(\ell^z_{T-t_1}-\ell^z_{T-t_2}\right)\right] \mu^+(\ud z).
\end{align*} 
Now, using Lemma \ref{lem:loc-time} in the above expression, we obtain \eqref{eq:lip-v}, after an application of Fubini's theorem.
\end{proof}

\begin{remark}
The condition on the transition density $\hat p$ in \eqref{eq:boundTD}, is perhaps more neatly expressed in terms of the standard transition density (with respect to the Lebesgue measure), denoted here by $p(\,\cdot\,)$. Indeed we notice that since $S'(y)=1$ 
\[
\hat p(s,x,y)=\frac{\hat p(s,x,y)}{S'(y)}=\tfrac{1}{2}\sigma^2(y)p(s,x,y).
\]
For many known transition densities we have that $p$ is uniformly bounded as soon as $s\in[\eps,\infty)$ for some $\eps>0$. Moreover, it is often the case that $p(s,x,\,\cdot\,)$ has an exponential decay at infinity (when $\cI$ is unbounded) so that mild growth conditions on $\sigma^2(y)\mu^+(\ud y)$ will guarantee \eqref{eq:boundTD}. 
\end{remark}

\begin{remark}
It is worth mentioning that Lipschitz continuity in time of the value function was also proved by \cite{JLL} using scaling properties of Brownian motion (in particular that for $s\in[0,T-t]$ one has $B_s=\sqrt{T-t}B_u$, with $u=s/(T-t)$). However,
for the argument in \cite{JLL} some additional regularity on $g$ and $\sigma$ is needed (e.g., local Lipschitz continuity of both functions).
\end{remark}

\begin{theorem}[{\bf $C^1$ time regularity}]\label{thm:C1-time}
Let Assumption \ref{ass:cont-v} hold and let $r,\sigma\in C^\alpha_{\ell oc}(\cI)$ for some $\alpha\in(0,1)$. If \eqref{eq:boundTD} holds with a constant $\kappa=\kappa(t,x)>0$ which is uniform for $(t,x)$ on compact subsets of $[0,T)\times\cI$, then $\partial_t v\in C([0,T)\times\cI)$.
\end{theorem}
\begin{proof}
For $(t,x)\in\mathrm{int}(\cS)$ we have $\partial_t v(t,x)=0$ and continuous (provided $\mathrm{int}(\cS)\neq\varnothing$). Corollary \ref{cor:PDE} guarantees that $\partial_t v$ is continuous in $\cC$ and therefore it remains to show that $\partial_t v$ is also continuous across the boundary $\partial\cC$. 

Fix $(t_0,x_0)\in\partial\cC$, with $t_0<T$, and take a sequence $(t_n,x_n)_{n\ge 1}\subset\cC$ such that $(t_n,x_n)\to (t_0,x_0)$ as $n\to \infty$. With no loss of generality we assume that $|x_n-x_0|\le \eta_0/2$ and $t_n<T-3\eps_0$ for all $n\ge 1$ and some $\eta_0\,,\eps_0>0$. Further, we denote $\cI_0:=(x_0-\eta_0,x_0+\eta_0)$. 

Next, let us derive an upper bound for $\partial_t v(t_n,x_n)$. Fix $n$ and take $\eps\in(0,\eps_0)$. 
Let $\tau_n=\tau_*^{t_n,x_n}$ be optimal for $v(t_n,x_n)$ and fix $s_0\in [0,\eps_0)$. Then $t_n+\eps_0+s_0<T-\eps_0$. Finally, set  
\[
\rho_n=\inf\{s\ge 0: X^{x_n}_s\notin \cI_0\}\wedge s_0.
\] 
By the (super)martingale property of $s\mapsto v(t_n+s,X^{x_n}_s)$ we have
\begin{align}\label{eq:vt0}
0\ge&\, v(t_n+\eps,x_n)-v(t_n,x_n)\\
\ge&\, \E_{x_n}\left[\e^{-\int_0^{\tau_n\wedge \rho_n}r(X_s)\ud s}\Big(v(t_n\!+\!\eps\!+\!\tau_n\wedge \rho_n,X_{\tau_n\wedge \rho_n})-v(t\!+\!\tau_n\wedge \rho_n,X_{\tau_n\wedge \rho_n})\Big)\right]\notag\\
=&\,\E_{x_n}\left[\ind_{\{\rho_n<\tau_n\}}\e^{-\int_0^{\rho_n}r(X_s)\ud s}\Big(v(t_n\!+\!\eps\!+\!\rho_n,X_{\rho_n})-v(t_n\!+\!\rho_n,X_{\rho_n})\Big)\right],\notag
\end{align}
where the final equality holds because $v(t_n\!+\!\eps\!+\!\tau_n,X_{\tau_n})=v(t_n\!+\!\tau_n,X_{\tau_n})=g(X_{\tau_n})$ on $\{\tau_n\le\rho_n\}$ by monotonicity of $t\mapsto v(t,x)$. Now, thanks to \eqref{eq:boundTD} we can find a constant $\kappa_0=\kappa(\cI_0,\eps_0)>0$, independent of $n$ and $s_0$, such that
\[
v(t_n\!+\!\eps\!+\!\rho_n,X_{\rho_n})-v(t_n\!+\!\rho_n,X_{\rho_n})\ge - \kappa_0\,\eps.
\]
Then, plugging the latter estimate into \eqref{eq:vt0}, recalling that $r\ge 0$, dividing by $\eps$ and letting $\eps\to 0$ we obtain
\begin{align}\label{eq:vt1}
0\ge \partial_t v(t_n,x_n)\ge -\kappa_0\P(\rho_n<\tau_n).
\end{align}

We are now interested in taking limits as $n\to\infty$ and showing that the right-hand side of \eqref{eq:vt1} goes to zero. First, let us rewrite 
\begin{align}\label{eq:vt1b}
\P(\rho_n<\tau_n)\le \P(\tau_n> s_0)+\P(\rho_n<s_0). 
\end{align}
From Proposition \ref{prop:reg-b} we know that $\P(\tau_n> s_0)\to 0$ as $n\to \infty$. We can estimate the second probability as follows. Define $\tilde\sigma$ as 
\begin{align*}
\tilde\sigma(x):=\left\{
\begin{array}{ll}
\sigma(x), & x\in \cI_0,\\[+4pt]
\sigma(x_0-\eta_0), & x\le x_0-\eta_0,\\[+4pt]
\sigma(x_0+\eta_0), & x\ge x_0+\eta_0, 
\end{array}
\right.
\end{align*}
along with the process $\widetilde X^n$ on $\R$, which is the unique (possibly weak) solution of 
\[
\ud \widetilde X^n_t=\tilde\sigma(\widetilde X^n_t)\ud B_t,\qquad \widetilde{X}^n_0=x_n.
\]
Existence of a unique in law, weak solution of the above SDE is guaranteed by Assumption \ref{ass:sigma1} and classical results (see \cite[Ch.\ 5.5]{KS}). 
By strong uniqueness of \eqref{eq:X} we also have $X^{x_n}_{t\wedge\rho_n}=\widetilde X^{n}_{t\wedge\tilde\rho_n}$ for all $t\ge 0$, $\P$-a.s., for $\tilde\rho_n=\inf\{t\ge 0: \widetilde X^{n}_t\notin\cI_0\}\wedge s_0$. 
Recall that $|x_0-x_n|<\eta_0/2$ for all $n\ge 1$. Therefore, using Markov inequality and Doob's martingale inequality we obtain 
\begin{align}\label{eq:lat}
\P\left(\rho_n< s_0\right)=&\P\left(\tilde\rho_n< s_0\right)\le \P\left(\sup_{0\le s\le s_0}\left|\int_0^s\tilde\sigma(\widetilde X^{n}_u)\ud B_u\right|\ge \frac{\eta_0}{2}\right)\\
\le& \frac{4}{\eta^2_0}\E\left[\sup_{0\le s\le s_0}\left|\int_0^s\tilde\sigma(\widetilde X^n_u)\ud B_u\right|^2\right]\le \frac{16}{\eta_0^2}\E\left[\int_0^{s_0}\tilde\sigma^2(\widetilde X^n_u)\ud u\right]\notag\\
\le & \frac{16}{\eta_0^2}\, s_0\,\sup_{x\in\cI_0}|\sigma(x)|=:\theta_0\, s_0,\notag
\end{align}
where the last inequality uses that $\sup_{x\in\R}|\tilde\sigma(x)|=\sup_{x\in\cI_0}|\sigma(x)|$ by construction.

Finally, using \eqref{eq:vt1b}, \eqref{eq:lat} and Proposition \ref{prop:reg-b} in \eqref{eq:vt1} we obtain
\[
0\ge \lim_{n\to\infty}\partial_t v(t_n,x_n)\ge - s_0\, \kappa_0\,\theta_0.
\]
Since $s_0>0$ can be taken arbitrarily small, this concludes the proof. 
\end{proof}

Remarkably, the time derivative is continuous irrespective of the regularity of the function $g$. This is in line with \cite{DeAPe}, but a direct application of results therein is not straightforward due to the lack of smoothness of $g$.

\begin{remark}\label{rem:C1}
The H\"older-continuity assumption on $\sigma$ and $r$ is only needed to guarantee that $\partial_t v$ is continuous in $\cC$ by Corollary \ref{cor:PDE}. Thanks to Remark \ref{rem:loc-reg} we can state a local version of Theorem \ref{thm:C1-time} only requiring that $r,\sigma\in C^\alpha_{\ell oc}(\cJ)$ for some open subset $\cJ\subset \cI$. Under such assumption we obtain $\partial_t v\in C([0,T)\times\cJ)$.
\end{remark}
Continuity of $\partial_t v$ has important consequences for the spatial regularity of the value function as well. For $\alpha\in(0,1)$ we denote $C^\alpha(\overline {\cJ})$ the class of $\alpha$-H\"older continuous functions on the closure of a set $\cJ$.
\begin{corollary}\label{cor:C1-space}
Let Assumption \ref{ass:cont-v} hold and let $r,\sigma\in C^\alpha\big(\overline{\cJ}\big)$ for some $\alpha\in(0,1)$, with $\cJ\subset \cI$ open and $\overline{\cJ}\subset\cI$. If \eqref{eq:boundTD} holds with a constant $\kappa=\kappa(t,x)>0$ which is uniform for $(t,x)$ on compacts subsets of $[0,T)\times\cI$, then $\partial_{xx} v$ admits a unique continuous extension to $\overline{\cC\cap([0,T-\delta]\times\cJ)}$ for any $\delta>0$.
\end{corollary}
\begin{proof}
Continuity of $\partial_{xx}v$ on $\overline{\cC\cap([0,T-\delta]\times\cJ)}$ follows directly from \eqref{eq:PDEv} and continuity of both $\partial_t v$ and $v$ on $[0,T)\times\overline{\cJ}$.
\end{proof}

\section{Continuation bays and stopping spikes}\label{sec:bays}

In this section we begin the study of the fine geometric properties of the optimal boundary $\partial\cC$. In contrast with the case of a smooth gain function, i.e., $g\in C^2(\cI)$, in this section we show that the possible presence of atoms in the measure $\mu(\ud x)$ produces effects that cannot be observed in the more regular cases. These will be illustrated in Example \ref{ex:bays} and \ref{ex:spikes} below.

By Hahn-Jordan decomposition \cite[Ch.\ VI, Sec.\ 29]{Hal} we can find two measurable sets $\Lambda_+$ and $\Lambda_-=\cI\setminus\Lambda_+$ such that $\mu^+(E)=\mu(\Lambda_+\cap E)\ge 0$ and $\mu^-(E)=-\mu(\Lambda_-\cap E)\ge 0$ for any measurable set $E$.
It is somewhat expected that the stopping set should lie in $[0,T]\times\Lambda_-$, where accumulating local time in the formulation \eqref{eq:u} is costly. This result is known to hold when $g\in C^2(\cI)$ and below we present some extensions to our setting. 

We are going to need the next lemma.
\begin{lemma}\label{lem:tech}
Fix $(t,x)\in[0,T)\times\cI$ and $\eps>0$. Let $K\supset(x-\eps,x+\eps)$ be a compact, define $\overline \sigma_K=\sup_{y\in K }|\sigma(y)|$ and let 
\[
\tau_\eps=\inf\big\{s\ge 0:X_s\notin(x-\eps,x+\eps)\big\}\wedge(T-t).
\] 
Then, for any $\delta>0$ and $z\in(x-\eps,x+\eps)$ we have
$\E_{x}\big[\ell^z_{\delta\wedge\tau_\eps}\big]\le \overline\sigma_K\sqrt{\delta}$. 

Moreover, setting $\bar r=\sup_{y\in K}|r(y)|$ we also obtain
\begin{align}\label{eq:liminflt}
\liminf_{\delta\to 0}\delta^{-1/2}\E_{x}\big[\ell^{x}_{\delta\wedge\tau_\eps}\big]\ge \frac{1}{\sqrt{2\pi}}\underline \sigma_K\e^{-\bar r T}.
\end{align}
\end{lemma}
\begin{proof}
Let us set $\cI_\eps:=(x-\eps,x+\eps)$. Since $r\ge 0$, then $\E_{x}\left[ \ell^z_{\tau_\eps\wedge \delta}\right]\le \E_{x}\left[ L^z_{\tau_\eps\wedge \delta}\right]$.
Recalling Assumption \ref{ass:sigma1} and applying It\^o-Tanaka formula, triangular inequality and Jensen's inequality we easily obtain
\begin{align}\label{eq:loc1}
\E_{x}\left[ L^z_{\tau_\eps\wedge \delta}\right]
=&\E_{x}\bigg[\bigg|x-z+\int_0^{\tau_\eps\wedge \delta}\sigma(X_t)\ud B_t\bigg|-|x-z|\bigg]\\
\le& \E_{x}\bigg[\bigg|\int_0^{\tau_\eps\wedge \delta}\sigma(X_t)\ud B_t\bigg|\bigg]\le \bigg(\E_{x}\bigg[\bigg|\int_0^{\tau_\eps\wedge \delta}\sigma(X_t)\ud B_t\bigg|^2\bigg]\bigg)^{\frac{1}{2}}\notag\\
=&\bigg(\E_{x}\bigg[\int_0^{\tau_\eps\wedge \delta}\big|\sigma(X_t)\big|^2\ud t\bigg]\bigg)^{\frac{1}{2}}\le \overline{\sigma}_K\left(\E_{x}[\tau_\eps\wedge \delta]\right)^{\frac{1}{2}}\le \overline{\sigma}_K\sqrt{\delta},\notag
\end{align}
where $\overline{\sigma}_K=\sup_{y\in K}|\sigma(y)|$, with $K\subset \cI$ a compact set that contains $\cI_\eps$.

For \eqref{eq:liminflt} we repeat steps similar to those in a proof given in \cite[Lemma 15]{Pe19}, being careful about the various constants cropping up in our case. Denote
\[
M_t=\int_0^t\sigma(X_s)\ud B_s,\quad\text{and}\quad \langle M\rangle_{t}=\int_0^{t}\sigma^2(X_s)\ud s
\]
and notice that $M_{t\wedge\tau_\eps}=W_{\langle M\rangle_{t\wedge\tau_\eps}}$ by Dambis-Dubins-Schwarz theorem, where $W$ is another Brownian motion (analogous to \eqref{eq:mJ}). By continuity of $r$ we have 
\[
\int_0^{\tau_\eps}r(X_t)\ud t\le T\cdot\bar r
\] 
with $\bar r=\sup_{x\in K}r(x)$ and $K\subset \cI$ as above.  
Then, using It\^o-Tanaka's formula as in \eqref{eq:loc1} with $z=x$ we also have
\begin{align*}
\E_{x}\left[ \ell^{x}_{\tau_\eps\wedge \delta}\right]\ge \e^{-T\bar r}\E_{x}\left[ L^{x}_{\tau_\eps\wedge \delta}\right]=\e^{-T\bar r}\E_{x}|M_{\tau_\eps\wedge\delta}|=\e^{-T\bar r}\E_{x}|W_{\langle M\rangle_{\tau_\eps\wedge\delta}}|.
\end{align*}
Letting $\widetilde L^0$ denote the local-time at zero of the Brownian motion $W$, a further application of It\^o-Tanaka's formula and optional sampling gives (notice that $\langle M\rangle_{\tau_\eps\wedge\delta}\le \overline\sigma^2_K\delta$ by Assumption \ref{ass:sigma1})
\[
\E_{x}\big|W_{\langle M\rangle_{\tau_\eps\wedge\delta}}\big|=\E_{x}\big[\widetilde L^0_{\langle M\rangle_{\tau_\eps\wedge\delta}}\big]\ge \E_{x}\big[\widetilde L^0_{\underline \sigma^2_K(\tau_\eps\wedge\delta)}\big],
\]
where the final inequality uses that $\underline \sigma_K=\inf_{x\in K}\sigma(x)>0$ (Assumption \ref{ass:sigma1}) and monotonicity of the local time. Notice that the notation $\E_{x}[\,\cdot\,]=\E[\,\cdot\,|X_0=x]$ keeps track of the fact that $\tau_\eps=\tau_\eps^{x}=\inf\{s\ge 0: X^{x}_s\notin\cI_{\eps}\}\wedge(T-t)$ and that the Brownian motion $W=W^{x}$, obtained via time-change, also depends on $x$ through the quadratic variation $\langle M^{x}\rangle$.
Next we proceed with simple estimates:
\begin{align}\label{eq:lt0}
\E_{x}\left[\widetilde L^0_{\underline \sigma^2_K(\tau_\eps\wedge\delta)}\right]=&\,\E_{x}\left[\ind_{\{\tau_\eps\ge\delta\}}\widetilde L^0_{\underline \sigma^2_K\,\delta}+\ind_{\{\tau_\eps< \delta\}}\widetilde L^0_{\underline \sigma^2_K\tau_\eps}\right]\\
= &\,\E_{x}\left[\widetilde L^0_{\underline \sigma^2_K\,\delta}-\ind_{\{\tau_\eps< \delta\}}\Big(\widetilde L^0_{\underline \sigma^2_K\,\delta}-\widetilde L^0_{\underline \sigma^2_K\tau_\eps}\Big)\right]\notag\\
\ge &\,\E_{x}\left[\widetilde L^0_{\underline \sigma^2_K\,\delta}-\ind_{\{\tau_\eps< \delta\}}\widetilde L^0_{\underline \sigma^2_K\,\delta}\right].\notag
\end{align}
It is well-known that the next chain of equalities holds in law under $\P_{x}$:
\[
\widetilde L^0_{\underline\sigma^2_K\,\delta}=\sup_{0\le s\le\, \underline{\sigma}^2_K\delta}W_s=|W_{\underline{\sigma}^2_K\delta}|=\sqrt{\underline{\sigma}^2_K\delta}|W_1|.
\] 
Then, we have also
\begin{align*}
\E_{x}\left[\ind_{\{\tau_\eps<\delta\}}\widetilde L^0_{\underline\sigma^2_K\,\delta}\right]\le&\left(\E_{x}\left[\big(\widetilde L^0_{\underline\sigma^2_K\,\delta}\big)^2\right]\right)^{\frac{1}{2}}\sqrt{\P_{x}(\tau_\eps<\delta)}=\sqrt{\underline\sigma_K^2\delta}\sqrt{\P_{x}(\tau_\eps<\delta)},
\end{align*}
upon noticing that $\E_x|W_1|^2=1$ since the law of $W_1$ is independent of $x$.

Thus, we obtain from \eqref{eq:lt0} and the discussion above
\begin{align}\label{eq:loc2}
\E_{x}\left[ \ell^{x}_{\tau_\eps\wedge \delta}\right]\ge \sqrt{\underline{\sigma}^2_K\delta}\left(\E\left[|W_1|\right]-\sqrt{\P_{x}(\tau_\eps<\delta)}\right)\e^{-T\bar r}.
\end{align}
Since $\P_{x}(\tau_{\eps}>0)=1$, then we can find $\delta_0>0$ sufficiently small, so that 
\[
\E_{x}\left[ \ell^{x}_{\tau_\eps\wedge \delta}\right]\ge \tfrac{1}{2}\sqrt{\underline{\sigma}^2_K\delta}\E\left[|W_1|\right]\e^{-T\bar r}, \quad \text{for $\delta\in(0,\delta_0)$}.
\]
Then, \eqref{eq:liminflt} follows upon recalling also $\E[|W_1|]=\sqrt{2/\pi}$.
\end{proof}

As immediate consequence of the lemma we have the next result.
\begin{proposition}\label{prop:bay}
Let $x_0\in\cI$ be such that $\mu(\{x_0\})>0$. Then $[0,T)\times \{x_0\}\subset\cC$.
\end{proposition}
\begin{proof}
Fix $t\in[0,T)$. Set $\theta:=\mu(\{x_0\})>0$ and, for arbitrary $\eps>0$, denote $\cI_\eps=(x_0-\eps,x_0+\eps)$ and $\tau_\eps=\inf\{s\ge 0: X_s\notin\cI_\eps\}$. Notice that $\ell^z_{\tau_{\eps}}=0$, $\P_{x_0}$-a.s.\ for $z\notin\cI_\eps$. Notice also that $x_0\in\Lambda_+$ and therefore $\mu^-(\{x_0\})=0$. Since $\tau_\eps\wedge\delta$ is admissible for any $\delta\in(0,T-t]$, then we have
\begin{align}\label{eq:loc0}
u(t,x_0)\ge& \E_{x_0}\left[\tfrac{1}{2}\int_{\cI_\eps} \ell^z_{\tau_\eps\wedge \delta} \mu(\ud z)\right]\\
\ge& \tfrac{1}{2}\E_{x_0}\left[\left(\theta\, \ell^{x_0}_{\tau_\eps\wedge \delta}-\int_{\cI_\eps\setminus\{x_0\}} \ell^z_{\tau_\eps\wedge \delta} \mu^-(\ud z)\right)\right]\notag
\end{align}
From the first claim in Lemma \ref{lem:tech} we have
\begin{align}\label{eq:atom}
u(t,x_0)\ge \tfrac{1}{2}\theta\E_{x_0}[\ell^{x_0}_{\tau_\eps\wedge\delta}]-\tfrac{1}{2}\overline{\sigma}_K\mu^-(\cI_\eps\setminus\{x_0\})\sqrt{\delta},
\end{align}
where we notice that the compact $K$ can be taken independent of $\eps\in(0,\eps_{0})$ for a fixed $\eps_0$.
Since $\mu^-$ is continuous at $x_0$ we can pick $\eps_1\in(0,\eps_0)$ sufficiently small and such that 
\[
\tfrac{1}{2}\overline{\sigma}_K\mu^-(\cI_{\eps_1}\setminus\{x_0\})<\tfrac{1}{3\sqrt{2\pi}}\theta\underline{\sigma}_K\e^{-T\bar r}.
\]
Having fixed $\eps_1>0$, thanks to \eqref{eq:liminflt} we can find $\delta_1>0$ sufficiently small, so that the right-hand side of \eqref{eq:atom} is strictly positive. Hence we conclude that $u(t,x_0)>0$. Since $t\in[0,T)$ was arbitrary, then $[0,T)\times\{x_0\}\in\cC$.
\end{proof}

Let us now introduce suitable subsets of $\Lambda_\pm$ that will be useful to prove properties of $\cC$ and $\cS$. For $x\in\cI$ we denote by $\cO_x\subset\cI$ an open neighbourhood of $x$ and we set
\[
\Lambda^0_\pm=\{x\in\Lambda_\pm:\,\mu^\pm(\cO_x)>0\,\text{for all $\cO_x$}\}.
\]
\begin{proposition}\label{prop:stop1}
If $(a,b)\subset\Lambda^0_+$, then $[0,T)\times(a,b)\subseteq\cC$.
\end{proposition}
\begin{proof}
Take an arbitrary $x_0\in(a,b)$ and let $\eps>0$ be such that $\cI_\eps:=(x_0-\eps,x_0+\eps)\subset(a,b)$. Then, $\mu^+(\cI_\eps)>0$ with $\mu^-(\cI_\eps)=0$ because $(a,b)\subset\Lambda_+$. Fix $t\in[0,T)$ and let $\tau_\eps=\inf\{s\ge 0: X_s\notin \cI_\eps\}\wedge (T-t)$. Below we will use the following fact
\begin{align}\label{eq:ellprob}
\E_{x_0}[\ell^z_{\tau_\eps}]>0,\quad \text{for all $z\in \cI_\eps$}, 
\end{align}
whose proof we also provide in the Appendix for completeness.

The stopping time $\tau_\eps$ is admissible and sub-optimal for the stopping problem with starting point $(t,x_0)$ and $\P_{x_0}(\tau_\eps>0)=1$ by the continuity of paths of $X$. Then, using \eqref{eq:u} we obtain 
\[
u(t,x_0)\ge \E_{x_0}\left[\tfrac{1}{2}\int_\cI \ell^z_{\tau_\eps} \mu(\ud z)\right]=\E_{x_0}\left[\tfrac{1}{2}\int_{\cI_\eps} \ell^z_{\tau_\eps} \mu^+(\ud z)\right]>0,
\] 
where in the equality we used that $\ell^z_{\tau_\eps}=0$, $\P_{x_0}$-a.s.\ for $z\notin\cI_\eps$ and the final inequality is by \eqref{eq:ellprob} and Fubini's theorem. Since $u(t,x_0)>0$ then $(t,x_0)\in\cC$. Recalling that $t\in[0,T)$ can be chosen arbitrarily gives $[0,T)\times\{x_0\}\in\cC$. Since $x_0\in(a,b)$ was also arbitrary, then $[0,T)\times(a,b)\subset\cC$ as claimed.
\end{proof}

Next we show that $c(x)<T$ at points $x\in\Lambda^0_-$ and that the stopping set is connected in the sense of \eqref{eq:connS1} below. For that, it is convenient to recall continuity of the value function and for simplicity we will also require the integrability condition
\begin{align}\label{eq:intM}
\sup_{x\in K}\E_x\left[\sup_{0\le t\le T}\e^{-\int_0^t r(X_s)\ud s}\Big(\big|g(X_t)\big|+|X_t|\Big)\right]<+\infty.
\end{align} 
The latter strengthens slightly the requirement in \eqref{eq:integr-g} by adding uniform integrability of the discounted process $X$.
\begin{proposition}\label{prop:stop2}
Let Assumption \ref{ass:cont-v} hold. 
\begin{itemize}
\item[(i)]
If $(a,b)\subset\Lambda^0_-$ and $r,\sigma\in C^\alpha([a,b])$ for some $\alpha\in(0,1)$, then $c(x)<T$ for $x\in(a,b)$; moreover, for any $x_1<x_2$ in $(a,b)$ we have
\begin{align}\label{eq:connS1}
[c(x_1)\vee c(x_2),T]\times[x_1,x_2]\subseteq\cS.
\end{align}
\item[(ii)] Let condition \eqref{eq:intM} hold. If $x_0\in\cI$ is such that $\mu(\{x_0\})<0$, then $c(x_0)<T$.
\end{itemize}
\end{proposition}
\begin{proof}
We divide the proof into two steps.

{\bf Step 1}. ({\em Proof of} (i)). To prove the first statement we argue by contradiction. Let us first assume that $(a,b)\subset\Lambda^0_-$ and $c(x)=T$ for all $x\in(a,b)$. We use ideas as in \cite{DeA15} but without requiring smoothness of $\sigma$. Consider the rectangular domain $\cR:=[0,T)\times(a,b)\subset\cC$ with parabolic boundary $\partial_P\cR=([0,T]\times[\{a\}\cup\{b\}])\cup(\{T\}\times(a,b))$. By Corollary \ref{cor:PDE} and Remark \ref{rem:loc-reg} we know that $v$ is the unique solution of the boundary value problem
\begin{align}\label{eq:PDE0}
\partial_t w+\tfrac{\sigma^2}{2}\partial_{xx}w=rw,\quad\text{on $\cR$ with $w=v$ on $\partial_P\cR$}.
\end{align} 
Monotonicity of $t\mapsto v(t,x)$ and \eqref{eq:PDE0} imply 
\[
\partial_{xx}v=2\sigma^{-2}(rv-\partial_t v)\ge 2\sigma^{-2}rv\qquad\text{on $\mathcal R$}.
\]
Let $\varphi\in C^\infty_c(a,b)$ with $\varphi\ge 0$ and $\int_a^b\varphi(x)\ud x=1$, multiply both sides of the inequality above by $\varphi$ and integrate over $(a,b)$. Integration by parts gives 
\begin{align*}
0\le &\int_a^b\big(\partial_{xx}v(t,x)-2\frac{r(x)}{\sigma^2(x)}v(t,x)\big)\varphi(x)\ud x\\
 =& \int_a^b \big(v(t,x)\partial_{xx}\varphi(x)-2\frac{r(x)}{\sigma^2(x)}v(t,x)\varphi(x)\big)\ud x,\quad\text{for $t\in[0,T)$.}
\end{align*}
Letting $t\to T$ in the above we obtain
\begin{align*}
0\le& \lim_{t\to T}\int_a^b \big(v(t,x)\partial_{xx}\varphi(x)-2\frac{r(x)}{\sigma^2(x)}v(t,x)\varphi(x)\big)\ud x\\
=&\int_a^b \big(g(x)\partial_{xx}\varphi(x)-2\frac{r(x)}{\sigma^2(x)}g(x)\varphi(x)\big)\ud x,
\end{align*}
where the final equality uses dominated convergence, continuity of the value function and $v(T,x)=g(x)$. Undoing the integration by parts we reach a contradiction with
\[
0\le \int_{(a,b)}\varphi(x)\mu(\ud x)<0,
\]
where the strict inequality holds by arbitrariness of $\varphi$ and $(a,b)\subset\Lambda^0_-$.
Then 
\begin{align}\label{eq:nonempty}
\big([0,T)\times(a,b)\big)\cap\cS\neq\varnothing.
\end{align}
By the exact same argument we can show that $\cD:=\{x\in(a,b):\,c(x)<T\}$ cannot contain isolated points in $(a,b)$ (otherwise we could construct a suitable rectangle $\cR'$ and reach a contradiction as above). 

Next we show that \eqref{eq:connS1} holds for any two points in $\cD$. 
Let us argue by contradiction again: take any two points $x_1<x_2$ in $\cD$, set $t_0=c(x_1)\vee c(x_2)<T$ and assume there exists $x_3\in(x_1,x_2)$ such that $(t_0,x_3)\in\cC$. Then, $\tau_*^{t_0,x_3}\le \inf\{t\ge 0: X^{x_3}_t\notin (x_1,x_2)\}$ because the segments $[t_0,T]\times\{x_i\}$, $i=1,2$, lie in the stopping set (recall that $\tau_*^{t_0,x_3}$ is the first time $(t_0+s,X^{x_3}_s)$ enters $\cS$). Then, 
\begin{align}\label{eq:mu-}
u(t_0,x_3)=\E_{x_3}\left[\tfrac{1}{2}\int_\cI \ell^z_{\tau_*}\mu(\ud z)\right]=-\E_{x_3}\left[\tfrac{1}{2}\int_{[x_1,x_2]} \ell^z_{\tau_*}\mu^-(\ud z)\right]\le 0
\end{align}
gives us a contradiction and \eqref{eq:connS1} holds.

Since \eqref{eq:connS1} holds in $\cD$ and the latter set has no isolated points in $(a,b)$ we conclude that $c(x)<T$ for all $x\in(a,b)$. Indeed, assume by way of contradiction that there is $x\in(a,b)$ such that $c(x)=T$. There are $x_1,x_2\in\cD$ with $x_1<x<x_2$ and \eqref{eq:connS1} holds for such $x_1$ and $x_2$. Then we have reached a contradiction.
\vspace{+3pt}

{\bf Step 2}. ({\em Proof of} (ii)). It remains to prove the final statement. Let $x_0\in\cI$ be such that $\theta:=-\mu(\{x_0\})>0$. Fix $\eps_0>0$ and let $\cI_\eps:=(x_0-\eps,x_0+\eps)$ for any $\eps\in(0,\eps_0]$, with $\cI^0:=\cI_{\eps_0}$. Since $x_0\in\Lambda_-$, then $\mu^+(\{x_0\})=0$ and we have
\begin{align*}
u(t,x_0)=&\sup_{0\le \tau \le T-t}\E_{x_0}\left[\tfrac{1}{2}\left(-\theta\, \ell^{x_0}_\tau+\int_{\cI^0\setminus\{x_0\}}\ell^z_\tau \mu(\ud z)+\int_{\cI\setminus\cI^0}\ell^z_\tau \mu(\ud z)\right)\right]\\
\le &\sup_{0\le \tau \le T-t}\E_{x_0}\left[\tfrac{1}{2}\left(-\theta\, \ell^{x_0}_\tau+\int_{\cI^0\setminus\{x_0\}}\ell^z_\tau [1+\mu^+](\ud z)+\int_{\cI\setminus\cI^0}\ell^z_\tau \mu(\ud z)\right)\right]\\
=:&\, \tilde{u}(t,x_0).
\end{align*}
Here $\tilde u(t,x)$ is the value function of a stopping problem of the form \eqref{eq:u} but with $\mu$ replaced by the measure 
\begin{align*}
\tilde \mu(\ud x)=& \ind_{\cI\setminus\cI^0}\mu(\ud x)+\ind_{\cI^0\setminus\{x_0\}}[1+\mu^+](\ud x)-\theta\,\delta_{x_0}(\ud x)\\
=&\mu(\ud x)+\ind_{\cI^0\setminus\{x_0\}}[1+\mu^-](\ud x),
\end{align*}
with $\delta_{x_0}$ the Dirac's delta at $x_0$. Clearly this problem enjoys the same properties of the original one: the gain function $\tilde g$ associated to $\tilde \mu$ (i.e., $\tilde g''=\tilde \mu$) is difference of two convex functions and, up to an affine transformation, it can be chosen so that $0\le \tilde g-g\le c(1+|x|)$ for a suitable constant $c>0$ depending on $\cI^0$; hence \eqref{eq:integr-g} holds for $\tilde g$ thanks to \eqref{eq:intM}. Differently to $\mu(\ud x)$ the measure $\tilde \mu(\ud x)$ is strictly positive on $\cI^0\setminus\{x_0\}$. That is, $\cI^0\setminus\{x_0\}\subset\tilde{\Lambda}^0_+$, where $\tilde{\Lambda}^0_+$ is the analogue of $\Lambda^0_+$ for the measure $\tilde\mu$. Then, by the same argument as in the proof of Proposition \ref{prop:stop1} we have $[0,T)\times\big(\cI^0\setminus\{x_0\}\big)\subset\tilde{\cC}$, where we denote $\tilde{\cC}=\{(t,x):\tilde u(t,x) >0\}$. Notice that since $u\le \tilde u$, then $\cC\subset\tilde{\cC}$.

Now, arguing by contradiction we assume $[0,T)\times\{x_0\}\subset\cC$. Then  $[0,T)\times\{x_0\}\subset\tilde\cC$ and the latter implies also $[0,T)\times\cI^0\subset\tilde{\cC}$, by the discussion above. This will lead to a contradiction. Fix $t\in[0,T)$ and $\eps\in(0,\eps_0]$, and let 
\[
\tau^{x_0}_\eps=\inf\{s\ge 0:X^{x_0}_s\notin\cI_\eps\}\wedge(T-t).
\] 
Then, letting $\tilde\tau^{t,x_0}_*$ be the optimal stopping time for $\tilde u(t,x_0)$, we have $\P\big(\tau^{x_0}_\eps\le \tilde\tau^{t,x_0}_*\big)=1$ because $\cI_\eps\subset\cI^0$ and $[0,T)\times\cI^0\subset \tilde \cC$ by assumption. Using the martingale property \eqref{eq:mart} for the value function $\tilde u$ and noticing that $\tilde u(T,x)=0$ for all $x\in\cI$, we have
\begin{align*}
\tilde u(t,x_0)=& \E_{x_0}\left[\e^{-\int_0^{\tau_\eps}r(X_s)\ud s}\tilde u(t+\tau_\eps,X_{\tau_\eps})+\int_\cI \tfrac{1}{2}\ell^z_{\tau_\eps}\tilde \mu(\ud z)\right]\\
\le & c_0 \P_{x_0}(\tau_\eps<T-t)+\tfrac{1}{2}\E_{x_0}\left[-\theta\, \ell^{x_0}_{\tau_\eps}+\int_{\cI_\eps\setminus\{x_0\}} \ell^z_{\tau_\eps}\tilde \mu(\ud z)\right],
\end{align*}
where $c_0:=\sup_{[0,T]\times\cI^0}|\tilde{u}(t,x)|$ is finite thanks to \eqref{eq:bound-v} applied to $\tilde v:=\tilde u+\tilde g$. By the exact same arguments as in the proof of Lemma \ref{lem:tech} (see \eqref{eq:loc1} and \eqref{eq:loc2}) we obtain the upper bound
\begin{align*}
\tilde u(t,x_0)\le & c_0 \P_{x_0}(\tau_\eps<T-t)+\tfrac{1}{2}\overline{\sigma}_0\sqrt{T-t}\,\tilde \mu^+\big(\cI_\eps\setminus\{x_0\}\big)\\
&-\tfrac{1}{2}\theta\,\underline \sigma_0\sqrt{T-t}\left(\E|W_1|-\sqrt{\P_{x_0}(\tau_\eps< T-t)}\right)\e^{-T\bar r},
\end{align*}
where $\underline\sigma_0=\inf_{x\in\cI^0}\sigma(x)$, $\overline\sigma_0=\sup_{x\in\cI^0}\sigma(x)$ and $\bar r=\sup_{x\in\cI^0}r(x)$. Since $\mu^+(\{x_0\})=0$, then the same holds for $\tilde \mu^+$ and we can select $\eps_1\in(0,\eps_0]$ sufficiently small that 
\[
\overline\sigma_0\tilde \mu^+\big(\cI_{\eps_1}\setminus\{x_0\}\big)\le \tfrac{1}{2}\theta\,\underline\sigma_0\,\E|W_1|\e^{-T\bar r}.
\]
Hence,
\begin{align*}
\tilde u(t,x_0)\le & c_0 \P_{x_0}(\tau_{\eps_1}<T-t)-\tfrac{1}{2}\theta\,\underline\sigma_0\sqrt{T-t}\,\left(\tfrac{1}{2}\E|W_1|-\sqrt{\P_{x_0}(\tau_{\eps_1}< T-t)}\right)\e^{-T\bar r}.
\end{align*}
By continuity of paths of $X$ it is clear that $\P_{x_0}(\tau_{\eps_1}< T-t)\to 0$ as $t\to T$ so that we can let $\tfrac{1}{2}\E|W_1|-\sqrt{\P_{x_0}(\tau_{\eps_1}< T-t)}\ge \tfrac{1}{4}\E|W_1|$ in the limit. Then both terms on the right-hand side above go to zero, with the second term being strictly negative and vanishing as $\sqrt{T-t}$ when $t\to T$. Assume that $\P_{x_0}(\tau_{\eps_1}<T-t)\approx(T-t)$ as $t\to T$, then there exists $t_1<T$ such that $\tilde u(t,x_0)<0$ for $t\in[t_1,T)$ and we reach a contradiction with $[0,T)\times\{x_0\}\subset\cC\subset\tilde\cC$.

It remains to show that $\P_{x_0}(\tau_{\eps_1}<T-t)\approx (T-t)$ as $t\to T$. For that, we define $\tilde\sigma$ as 
\begin{align*}
\tilde\sigma(x):=\left\{
\begin{array}{ll}
\sigma(x), & x\in \cI^0,\\[+4pt]
\sigma(x_0-\eps_0), & x\le x_0-\eps_0,\\[+4pt]
\sigma(x_0+\eps_0), & x\ge x_0+\eps_0, 
\end{array}
\right.
\end{align*}
along with the process $\widetilde X$ on $\R$, which is the unique (possibly weak) solution of
\[
\ud \widetilde X_t=\tilde\sigma(\widetilde X_t)\ud B_t,\qquad \widetilde{X}_0=x_0.
\]
By strong uniqueness of \eqref{eq:X} we have $X_{t\wedge\tau_{\eps_1}}=\widetilde X_{t\wedge\tilde\tau_{\eps_1}}$ for all $t\ge 0$, $\P_{x_0}$-a.s., for $\tilde\tau_{\eps_1}=\inf\{s\ge 0: \widetilde X_t\notin\cI_{\eps_1}\}\wedge(T-t)$. Therefore, using Markov inequality and Doob's martingale inequality we obtain 
\begin{align*}
\P_{x_0}\left(\tau_{\eps_1}< T-t\right)=&\,\P_{x_0}\left(\tilde\tau_{\eps_1}< T-t\right)\le \P_{x_0}\left(\sup_{0\le s\le T-t}\left|\int_0^s\tilde\sigma(\widetilde X_u)\ud B_u\right|\ge \frac{\eps_1}{2}\right)\\
\le& \frac{4}{\eps^2_1}\E_{x_0}\left[\sup_{0\le s\le T-t}\left|\int_0^s\tilde\sigma(\widetilde X_u)\ud B_u\right|^2\right]\le \frac{16}{\eps_1^2}\E_{x_0}\left[\int_0^{T-t}\tilde\sigma^2(\widetilde X_u)\ud u\right]\\
\le & \frac{16}{\eps_1^2}\overline{\sigma}^2_0(T-t),
\end{align*}
which concludes the proof.
\end{proof}

\begin{remark}[Flatness of $x\mapsto c(x)$]\label{re:flat-c}
The argument we used in step 1 of the proof above to obtain \eqref{eq:nonempty} was originally designed in \cite{DeA15} to show continuity of optimal boundaries as functions of time. Here, as a byproduct of the proof we obtain that the map $x\mapsto c(x)$ cannot exhibit a flat stretch, which is also strictly positive, on $\Lambda^0_-$. That is, if there exists an interval $(x_1,x_2)\subseteq \Lambda^0_-$ such that $c(x)=\hat c$ for $x\in(x_1,x_2)$, then it must be $\hat c=0$.
The proof is an exact repetition of the one for \eqref{eq:nonempty}, so we omit it.
\end{remark}

There is a nice monotonicity result that follows as a corollary from Proposition \ref{prop:stop2}. First notice that given any interval $[a,b]\subset\cI$ the boundary attains a global minimum on $[a,b]$ by lower semi-continuity. Then we can define the set of minimisers
\[
\Sigma_{[a,b]}:=\mathrm{argmin}\{c(x),\,x\in[a,b]\}
\]
and $\Sigma_{[a,b]}\neq\varnothing$ for any $a\le b$. Notice that $\Sigma_{[a,b]}$ is a closed set by lower semi-continuity of the boundary.
\begin{corollary}\label{cor:monot}
Let Assumption \ref{ass:cont-v} hold, let $(a,b)\subset\Lambda^0_-$ and assume $r,\sigma\in C^\alpha([a,b])$ for some $\alpha\in(0,1)$. Then $\Sigma_{[a,b]}=[a_*,b_*]$ for some $a_*\le b_*$ (with $\Sigma_{[a,b]}=a_*$ if $a_*=b_*$). Moreover, if $a_*<b_*$ then $c(x)=0$ on $[a_*,b_*]$. Finally, $x\mapsto c(x)$ is strictly decreasing on $[a,a_*)$ and strictly increasing on $(b_*,b]$ (with $[x,x)=(x,x]=\varnothing$). 
\end{corollary}
\begin{proof}
Let $x_1<x_2$ belong to $\Sigma_{[a,b]}$. Then $c(x_1)=c(x_2)=:\bar c$ and \eqref{eq:connS1} implies that $[\bar c,T]\times[x_1,x_2]\subset\cS$. Since $\bar c$ is the global minimum, then $c(x)=\bar c$ for all $x\in[x_1,x_2]$. Hence $[x_1,x_2]\in\Sigma_{[a,b]}$ and  since $x_1,x_2$ were arbitrary and $\Sigma_{[a,b]}$ is closed we conclude that $\Sigma_{[a,b]}=[a_*,b_*]$ for some $a_*\le b_*$. If $a_*<b_*$ then $c(x)=\bar c$ on $[a_*,b_*]$ and this can only occur if $\bar c=0$ (Remark \ref{re:flat-c}). 

For the final claim, assume $[a,a_*)\neq\varnothing$ and, arguing by contradiction, that there exist $x_1<x_2$ in $[a,a_*)$ such that $c(x_1)\le c(x_2)$. By definition of $a_*$ we have $c(x_1)>c(a_*)$. Then
$[c(x_1),T]\times[x_1,a_*]\subset \cS$ by \eqref{eq:connS1} and it must be $c(x_2)=c(x_1)=:\hat c$. By the same argument there cannot exist $x_3\in(x_1,x_2)$ such that  $c(x_3)<c(x_2)$ and therefore we conclude that $c(x)=\hat c$ for all $x\in[x_1,x_2]$. From Remark \ref{re:flat-c} we know that $x\mapsto c(x)$ cannot be flat, unless it is equal to zero. However, $x_2<a_*$ and therefore $\hat c=c(x_2)>c(a_*)\ge 0$. Thus we have reached a contradiction and $x\mapsto c(x)$ is strictly decreasing on $[a,a_*)$. By the same argument we can prove that the boundary is strictly increasing on $(b_*,b]$.
\end{proof}

Proposition \ref{prop:bay} holds at any point $x_0$ such that $\mu(\{x_0\})>0$, irrespective of the sign of $\mu(\ud x)$ in a neighbourhood of $x_0$. We will see in the next example that this argument, combined with Proposition \ref{prop:stop2}, can produce very peculiar shapes of the continuation set. Loosely speaking we can say that we find a {\em continuation bay} in the middle of a stopping set. 

\begin{figure}[t!]
\includegraphics[scale=0.4]{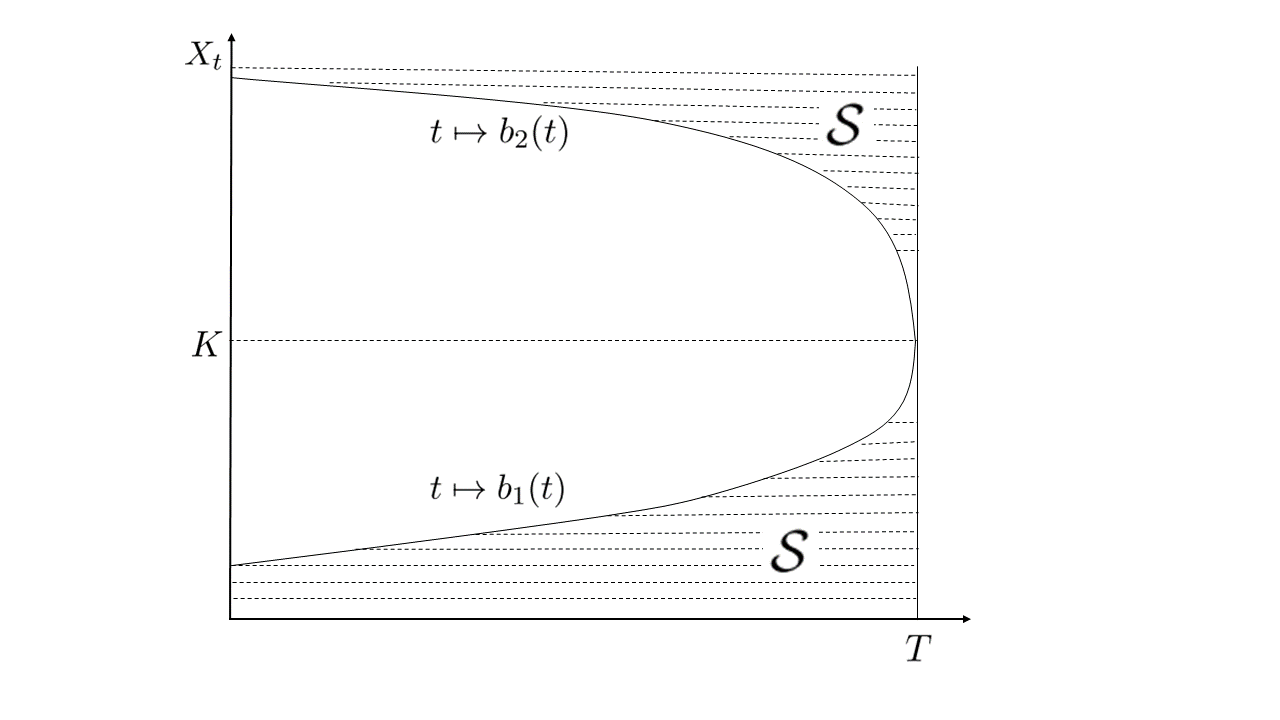}
\caption{An illustration of the continuation bay in Example \ref{ex:bays}.}
\end{figure}

\begin{example}[Continuation bays]\label{ex:bays}
{\rm
A typical example of continuation bay arises in the American straddle option (see, e.g., \cite{DE09}). Let us consider a simplified version here and let
\[
\ud X_t=\sigma X_t\, \ud B_t,\qquad X_0=x,
\] 
be the stock's dynamics with $\sigma>0$. Fix $K>0$ and $r>0$ and let us denote the value of the option by
\[
v(t,x)=\sup_{0\le \tau\le T-t}\E_x\left[\e^{-r\tau}|X_\tau-K|\right].
\]
Then, by an application of It\^o-Tanaka's formula we have 
\begin{align*}
v(t,x)-|x-K|=&\sup_{0\le \tau\le T-t}\E_x\left[\int_0^\tau \e^{-rt}\ud L^K_t-r\int_0^\tau \e^{-rt}|X_t-K|\ud t\right]\\
=&\sup_{0\le \tau\le T-t}\E_x\left[\tfrac{1}{2}\int_{\mathbb{R}_+}\ell^z_\tau \mu(\ud z)\right],
\end{align*}
where $\mu(\ud z)=2\delta_K(\ud z)-(2r/\sigma^2) z^{-2}|z-K|\ind_{\{z\neq K\}}\ud z$. 

Here we have $\Lambda_-=\R_+\setminus\{K\}=\Lambda^0_-$ and $\Lambda_+=\{K\}=\Lambda^0_+$, which is a rather `singular' situation. Intuitively, waiting is costly for the option holder at all times $t\in[0,T]$ for which $X_t\neq K$: indeed, she pays a cost at a rate $r|X_t-K|\ud t$. On the contrary, waiting is rewarding only at times $t\in[0,T]$ when $X_t=K$ and the option holder receives a reward at the `rate' of $\ud L^K_t$. As we will see shortly, it is precisely the {\em kink} in the payoff $x\mapsto|x-K|$ that guarantees $\cC\neq\varnothing$ and makes the problem mathematically non-trivial.

From (i) in Proposition \ref{prop:stop2} we obtain that $c(x)<T$ for all $x\in\R_+\setminus\{K\}$, whereas Proposition \ref{prop:bay} guarantees $c(K)=T$. By the same arguments we used to prove \eqref{eq:connS1} we can also show that for any $x>K$ we have $[c(x),T]\times [x,\infty)\in\cS$. Indeed, assume by contradiction that there exists $x'>x$ such that $(t,x')\in\cC$ for $t= c(x)$; then, $\tau^{t,x'}_*\le \inf\{s\ge 0: X^{x'}_s\le x\}$ and we obtain the analogue of \eqref{eq:mu-} with $(t_0,x_3)=(t,x')$ and $[x_1,x_2]$ replaced by $[x,\infty)$. Hence a contradiction. Likewise, we can show that for any $x\in(0,K)$, we have $[c(x),T]\times [0,x]\in\cS$. Finally, Corollary \ref{cor:monot} implies that  $c$ is strictly increasing on $(0,K)$ and strictly decreasing on $(K,\infty)$, hence it can be inverted (locally) defining two boundaries which are continuous functions of time. Indeed, let $c_1(x)=c(x)$ for $x\in(0,K)$ and $c_2(x)=c(x)$ for $x>K$, then we can set
\[
b_1(t):= c^{-1}_1(t)\qquad \text{and}\qquad b_2(t):=c^{-1}_2(t),\quad t\in[0,T).
\]
The functions $b_1$ and $b_2$ are continuous with $b_1(T)=b_2(T)=K$. 
It may be worth noticing that a one-sided version of continuation bay appears by the same argument also in the American put and call options.
}
\end{example}

A reverse situation is observed at points $x_0$ such that $\mu(\{x_0\})<0$. In this case, if $\mu(\ud x)>0$ on a neighbourhood of $x_0$, we observe a {\em stopping spike} in the middle of the continuation region. This type of geometry of the stopping set is almost unique and certainly not very popular in the literature. The only examples of a similar geometry that we are aware of appear in \cite{Ped03} and \cite{DK18} but the settings are different: in both references the gain function is time-dependent and in \cite{Ped03} it is discontinuous in the spatial variable whereas in \cite{DK18} it is discontinuous in the time variable. So it is difficult to draw a clear parallel. More closely related is the situation of game call options where, for some parameter choice, the option seller will only stop if the underlying asset's value equals the strike price (see, e.g., \cite{EV, Em, YYZ}).

This time we need to recall (ii) from Proposition \ref{prop:stop2}. 

\begin{figure}[h!]
\includegraphics[scale=0.4]{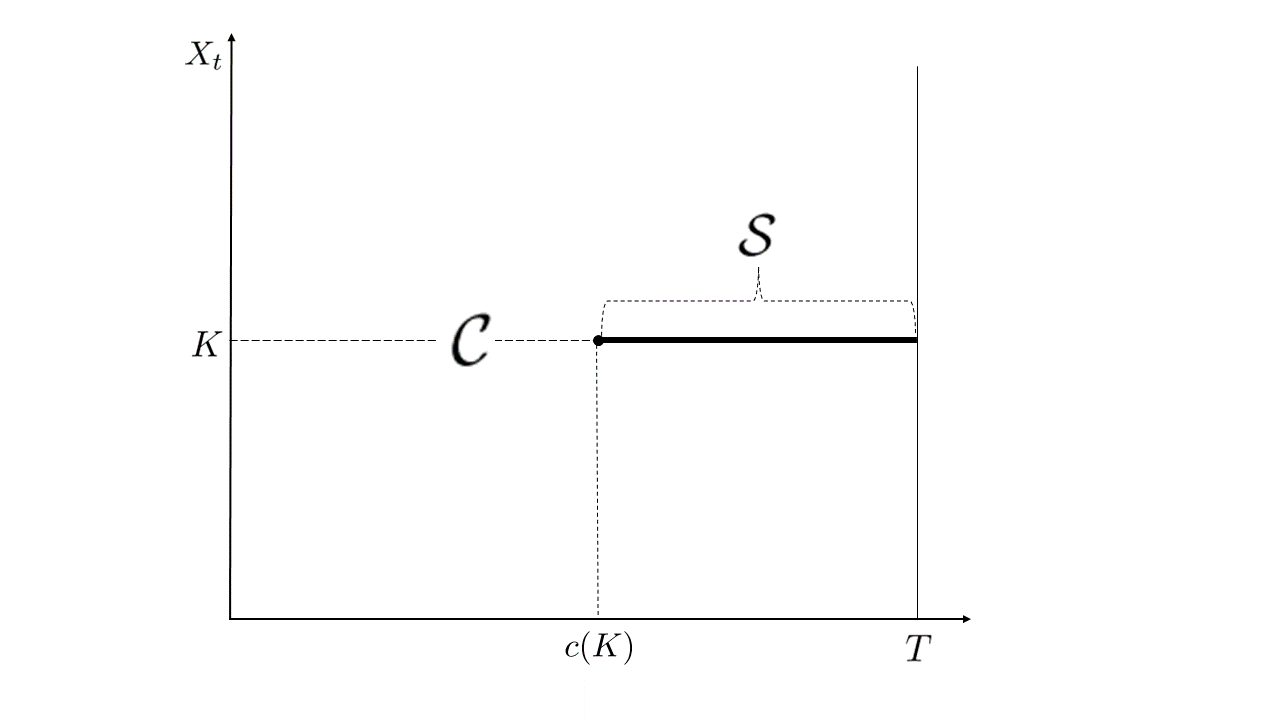}
\caption{An illustration of the stopping spike in Example \ref{ex:spikes}.}
\end{figure}

\begin{example}[Stopping spikes]\label{ex:spikes}
{\rm
For simplicity we consider a converse to Example \ref{ex:bays}. That is, we take
\[
\ud X_t=\sigma X_t\, \ud B_t,\qquad X_0=x,
\] 
and, for a fixed $\eta_0>0$, we consider 
\[
\hat v(t,x)=\inf_{0\le \tau\le T-t}\E_x\left[\e^{-r\tau}\left(|X_\tau-K|+\eta_0\right)\right].
\]
Here the stopper may be the seller of a cancellable straddle option of European type, who must pay a fee of $\eta_0$ (in addition to the option's current payoff) in order to cancel the contract. Although the problem is stated as a minimisation, it is clear that it is equivalent to
\[
-\hat v(t,x)=v(t,x)=\sup_{0\le \tau\le T-t}\E_x\left[\e^{-r\tau}\left(-|X_\tau-K|-\eta_0\right)\right]
\]
so that our arguments apply directly with $g(x)=-|x-K|-\eta_0$. In particular, by the same calculations as in Example \ref{ex:bays} we obtain $\Lambda_+=\R_+\setminus\{K\}=\Lambda^0_+$, which gives us 
\[
[0,T)\times\big(\R_+\setminus\{K\}\big)\subset\cC.
\]
So $\cS\subseteq[0,T]\times\{K\}$ and, by (ii) in Proposition \ref{prop:stop2}, we know that $c(K)<T$. Hence,  
\[
\cS=[c(K),T]\times\{K\}
\]
is just a {\em spike} in the continuation region. 

Notice that, due to discounting and to the presence of a cancellation fee $\eta_0>0$, if $T$ is sufficiently large we expect $c(K)>0$, as stopping at $K$ is not necessarily optimal if the time to maturity is long.}
\end{example}

\section{Continuity of the boundary}\label{sec:cont}
Here we address the question of continuity of the map $x\mapsto c(x)$ and its link to strict monotonicity of time-dependent optimal stopping boundaries. For $\alpha\in(0,1)$ we denote $C^{1,\alpha}([a,b])$ the class of continuously differentiable functions on $[a,b]$ whose first order derivative is also $\alpha$-H\"older continuous on $[a,b]$. Throughout the section we need to invoke Theorem \ref{thm:C1-time} several times, so we state the next assumption:
\begin{assumption}\label{ass:const}
The bound in \eqref{eq:boundTD} holds with a constant $\kappa=\kappa(t,x)>0$ which is uniform for $(t,x)$ on compact subsets of $[0,T)\times\cI$.
\end{assumption}

\begin{theorem}\label{thm:non-flat}
Let Assumptions \ref{ass:cont-v} and \ref{ass:const} hold. If $(a,b)\subset\Lambda^0_-$ and $\sigma,r,g\in C^{1,\alpha}([a,b])$ for some $\alpha\in(0,1)$, then $x\mapsto c(x)$ is continuous on $(a,b)$.
\end{theorem}
The proof of the theorem hinges on the following two lemmas. 
\begin{lemma}\label{lem:ut}
Let Assumption \ref{ass:cont-v} and \ref{ass:const} hold. If $r,\sigma\in C^\alpha([a,b])$ for some $\alpha\in(0,1)$, then $\partial_t v(t,x)<0$ for all $(t,x)\in \cC$ with $x\in(a,b)$.
\end{lemma}
The proof is essentially an application of the maximum principle and we give it in Appendix for completeness. 
\begin{lemma}\label{lem:sm-f}
Let Assumption \ref{ass:cont-v} and \ref{ass:const} hold and let $(a,b)\subset\Lambda^0_-$. If $\sigma, r, g\in C^{1,\alpha}([a,b])$ for some $\alpha\in(0,1)$ then $\partial_x v$ is continuous on $(\hat c,T)\times(a,b)$ where $\hat c=\min_{x\in[a,b]}c(x)$.
\end{lemma}
The proof is inspired by \cite{DeAPe} but we cannot directly invoke any of the results therein due to the local nature of our assumptions. However, if we strengthen the requirements in the lemma to, e.g., $\sigma,r,g\in C^1_b(\cI)$, then \cite[Theorem 10]{DeAPe} applies directly yielding $v\in C^1([0,T)\times\cI)$. In several practical applications Lemma \ref{lem:sm-f} may be better suited and therefore we give a full proof in Appendix.

\begin{proof}[Proof of Theorem \ref{thm:non-flat}]
Since $x\mapsto c(x)$ changes its monotonicity at most once on $(a,b)$ (Corollary \ref{cor:monot}) and it is lower semi-continuous, we only need to rule out discontinuities of the first kind. In particular, with no loss of generality we may assume that $c$ is strictly increasing on $(a,b)$ as the argument is analogous for decreasing boundaries and combining the two we can handle the general case.

First we notice that since $c$ is strictly increasing and lower semi-continuous on $(a,b)$, then it must be left-continuous. It then remains to prove that it is also right-continuous.
Arguing by contradiction let us assume that there exists $x_0\in[a,b]$ such that $c(x_0)<c(x_0+)$. Then $(c(x_0),c(x_0+))\times\{x_0\}\subset\partial\cC$ and there exists $x_1>x_0$ and $\eps_1>0$ such that 
\begin{align}\label{eq:vt-e1}
\partial_t v(t,x_1)\le -\eps_1\quad\text{for all $t\in[t_0,t_1]\subset(c(x_0),c(x_0+))$}
\end{align}
thanks to Lemma \ref{lem:ut} and the fact that $\partial_t v$ is continuous.
Setting $\hat c=\min_{x\in[a,b]}c(x)$ and combining Lemma \ref{lem:sm-f} and Theorem \ref{thm:C1-time} (recall also Remark \ref{rem:C1}) we can also conclude that $v\in C^1((\hat c,T)\times (a,b))$. Then for any $\eps>0$ there exists $\delta_{\varepsilon}>0$ such that $x_0+\delta_\eps<x_1$ and
\begin{align}\label{eq:bdd}
0 \ge\partial_t v \ge -\varepsilon\quad\text{and}\quad |\partial_x v-\partial_x g| \leq \varepsilon\quad
\text{on $[t_0,t_1] \times [x_0,x_0+\delta_{\varepsilon}]$,}
\end{align}
by uniform continuity on any compact. Classical results on interior regularity for solutions of PDEs guarantee $\partial_t v\in C^{1,2}\big((t_0,t_1)\times(x_0,x_1)\big)$ and
\begin{align}\label{eq:PDEvt}
\partial_{tt}v(t,x)+\frac{\sigma^2(x)}{2}\partial_{xxt} v(t,x)=r(x)\partial_t v(t,x),\quad(t,x)\in (t_0,t_1)\times(x_0,x_1)
\end{align}
(see, e.g., \cite[Thm.~10, Ch.~3, Sec.~5]{Friedman}).

Since $v=g$ on $\big(c(x_0),c(x_0+)\big)\times\{x_0\}$, we may expect that $v_{tx}$ be continuous at $(c(x_0),c(x_0+))\times\{x_0\}$ and equal to zero. From a PDE perspective that would enable the use of Hopf's lemma to reach a contradiction. Here instead we present a probabilistic analogue based on the construction of a process which is normally reflected `near' the discontinuity of the boundary. This approach avoids to deal with continuity {\em up to the optimal boundary} of the value function's derivatives of order greater than one.

On the interval $[x_0+\delta_{\eps},x_1)$ we consider a process that is equal to $(X_t)_{t \ge 0}$ away from $x_0+\delta_{\varepsilon}$, it is reflected (upwards) at $x_0+\delta_{\varepsilon}$ and it gets absorbed at $x_1$. For the construction of such process we extend the diffusion coefficient $\sigma$ outside $(a,b)$ to be $C^1_b(\R)$ and strictly separated from zero. With a slight abuse of notation let us denote such extension again by $\sigma$. Then, it is well-known (see, e.g., \cite{LS84} or \cite[Sec.~12, Chapter I]{Bass}) that there exists a unique strong solution of the stochastic differential equation
\begin{align*}
\ud \XX_t^{\varepsilon}=\sigma(\XX_t^{\varepsilon}) \ud B_t+\ud R^{\delta_\eps}_t, \qquad\qquad
\XX_{0}^{\varepsilon} = x_0+\delta_{\varepsilon},
\end{align*}
where $R^{\delta_{\varepsilon}}$ is a continuous, non-decreasing process, with $R^{\delta_\eps}_0=0$, that guarantees
\begin{align}\label{eq:SK}
&\XX_t^{\varepsilon} \ge x_0+\delta_{\varepsilon}\quad\text{and}\quad \ud R_{t}^{\delta_{\varepsilon}}= \ind_{\{\XX^\eps_t=x_0+\delta_\eps\}}\ud R_{t}^{\delta_{\varepsilon}}\quad\text{for all $t\ge 0$, $\P$-a.s.}
\end{align}
Setting $\tau^\eps_1=\inf\{s\ge0\,:\,\XX^{\varepsilon}_s =x_1\}$ the absorbed process is obtained as $(\XX^\eps_{t\wedge \tau^\eps_1})_{t\ge 0}$.

Letting $\hat v:= \partial_t v$ we can apply It\^o's formula for semi-martingales and use \eqref{eq:PDEvt} to obtain, for any $t\in(t_0,t_1)$
\begin{align}\label{eq:cont0}
\E&\! \left[\e^{-\int_0^{\tau^\eps_1\wedge(t_1-t)}r(\XX^\eps_s)\ud s}\, \hat v\big(t+\tau^\eps_1\wedge(t_1-t),\XX_{\tau^\eps_1\wedge(t_1-t)}^{\varepsilon}\big)\right]\notag\\
    =&\, \hat v(t,x_0+\delta_{\eps})+\E\left[\int_0^{\tau^\eps_1\wedge(t_1-t)}\e^{-\int_0^{s}r(\XX^\eps_u)\ud u}\,\partial_x \hat v(t+s,\XX_{s}^{\eps}) \,\ud R^{\delta_\eps}_{s}\right]\\
\ge& - \eps +\E\left[\int_0^{\tau^\eps_1\wedge(t_1-t)}\e^{-\int_0^{s}r(\XX^\eps_u)\ud u}\,\partial_{tx}v(t+s,x_0+\delta_{\varepsilon})\, \ud R_{s}^{\delta_{\varepsilon}}\right]\notag
\end{align}
where the inequality follows from \eqref{eq:bdd} and for the term under expectation we use \eqref{eq:SK}.
For the expression on the left-hand side of \eqref{eq:cont0}, denoting $\bar r=\sup_{x\in[x_0,x_1]}r(x)$ and recalling that $\partial_t v\le 0$, thanks to \eqref{eq:vt-e1} we have
\begin{align*}
&\E \left[ \e^{-\int_0^{\tau^\eps_1\wedge(t_1-t)}r(\XX^\eps_s)\ud s}\, \hat v\big(t+\tau^\eps_1\wedge(t_1-t),\XX_{\tau^\eps_1\wedge(t_1-t)}^{\varepsilon}\big)\right]\\
&\leq\e^{-\bar r\, T} \E\left[\ind_{\{\tau^\eps_1<t_1-t\}}\hat v(t+\tau^\eps_1,x_1)\right] \leq -\eps_1\,\e^{-\bar r\, T} \P (\tau^\eps_1<t_1-t).
\end{align*}
Hence, setting for simplicity $\bar \eps_1=\eps_1\e^{\bar r\,T}$, from \eqref{eq:cont0} we obtain
\begin{align}\label{eq:cont1}
&-\bar \eps_1\, \P (\tau^\eps_1<t_1-t)\\
& \ge - \eps +\E \left[\int_0^{\tau^\eps_1\wedge(t_1-t)}\e^{-\int_0^{s}r(\XX^\eps_u)\ud u}\,\partial_{tx}v(t+s,x_0+\delta_{\varepsilon})\, \ud R_{s}^{\delta_{\varepsilon}}\right].\notag
\end{align}

The next step is to let $\eps\to 0$. In order to take care of possible issues with the regularity of $\partial_{tx}v$ as $\delta_\eps\downarrow 0$ we adopt an approach using test functions. Pick a non-negative function $\varphi \in C^{\infty}_{c}(t_0,t_1)$ such that
$\int_{t_0}^{t_1}\varphi(t) \ud t = 1$. Then, multiplying both sides of \eqref{eq:cont1} by $\varphi$, integrating over $(t_0,t_1)$ and using Fubini's theorem we obtain
\begin{align}\label{eq:cont-01}
-\bar \eps_1 & \int_{t_0}^{t_1} \P(\tau^\eps_1<t_1-t)\varphi(t) \ud t\\
& \ge - \eps +\E\left[\int_0^{\tau^\eps_1}\e^{-\int_0^{s}r(\XX^\eps_u)\ud u}\left(\int_{t_0}^{t_1} \ind_{\{t<t_1-s\}}\partial_{tx}v(t+s,x_0+\delta_{\varepsilon})\varphi(t) \ud t \right) \ud R_{s}^{\delta_{\varepsilon}}\right],\notag
\end{align}
where we are also using that $\tau_1^\eps$ is independent of $t$. Let us now look more closely at the integral on the right-hand side above: integration by parts and the second estimate in \eqref{eq:bdd} give
\begin{align*}
&\int_{t_0}^{t_1}\! \ind_{\{t<t_1-s\}}\partial_{tx}v(t\!+\!s,x_0\!+\!\delta_{\varepsilon})\varphi(t) \ud t\\
&=\partial_x v(t_1,x_0\!+\!\delta_{\varepsilon})\varphi(t_1-s)-\!\int_{t_0}^{t_1}\! \ind_{\{t<t_1-s\}}\partial_x v(t\!+\!s,x_0\!+\!\delta_{\varepsilon})\varphi'(t) \ud t\\
&\ge \big(\partial_x g(x_0\!+\!\delta_{\varepsilon})\!-\!\eps\big)\varphi(t_1-s)\!-\!\int_{t_0}^{t_1}\! \ind_{\{t<t_1-s\}}\partial_x g (x_0\!+\!\delta_{\varepsilon})\varphi'(t) \ud t\!-\!\eps\!\int_{t_0}^{t_1}\! \ind_{\{t<t_1-s\}}|\varphi'(t)| \ud t\\
&=-\eps\varphi(t_1-s)-\!\eps\!\int_{t_0}^{t_1}\! \ind_{\{t<t_1-s\}}|\varphi'(t)| \ud t,
\end{align*}
where the final equality follows by integrating $\varphi'$ over $(t_0,t_1)$. Using the expression above in \eqref{eq:cont-01} along with $r(\,\cdot\,)\ge 0$ we obtain
\begin{align*}
-\bar \eps_1 & \int_{t_0}^{t_1} \P(\tau^\eps_1<t_1-t)\varphi(t) \ud t\\
& \ge - \eps -\eps\, \E\bigg[\int_0^{\tau^\eps_1}\varphi(t_1-s) \ud R_{s}^{\delta_{\varepsilon}}+R^{\delta_\eps}_{\tau^\eps_1\wedge t_1}\int_{t_0}^{t_1}|\varphi'(t)|\ud t\bigg] \\
&\ge -\eps\left(1+\left(\|\varphi \|_{\infty}+T\|\varphi' \|_{\infty}\right)\E\big[R_{\tau^\eps_1\wedge t_1}^{\delta_{\varepsilon}}\big]\right),
\end{align*}
where the final inequality uses that $\varphi(t_1-s)=0$ for $s\ge t_1$ and $\|\cdot\|_\infty$ is the supremum norm on $[0,T]$.

From the integral form of the dynamics of $\XX^\eps$ we obtain
\begin{align*}
\E\left[R^{\delta_{\varepsilon}}_{\tau^\eps_1\wedge t_1}\right]=\E\left[\XX_{\tau^\eps_1\wedge t_1}^{\eps}
 -x_0-\delta_{\eps}\right]\le x_1-x_0
\end{align*}
and 
\begin{align*}
-\bar \eps_1 & \int_{t_0}^{t_1} \P(\tau^\eps_1<t_1-t)\varphi(t) \ud t \ge -\eps\Big(1+\left(\|\varphi \|_{\infty}+T\|\varphi' \|_{\infty}\right)(x_1-x_0)\Big).
\end{align*}
Then, taking limits as $\eps\to 0$ gives
\begin{align}\label{eq:posit}
\limsup_{\eps\to 0} \int_{t_0}^{t_1} \P(\tau^\eps_1<t_1-t)\varphi(t) \ud t \le 0.
\end{align}
Showing that the left hand side above is positive will give us a contradiction. Hence there cannot be a discontinuity of $c$ at $x_0$.

Setting $\cJ=(a,x_1)$ and adopting the same time-change as in step 2 of the proof of Proposition \ref{prop:reg-b} (see \eqref{eq:mJ} and \eqref{eq:XZ}) we obtain, using the same notation,
\[
\XX^\eps_{A_s\wedge\tau_\cJ}=x_0+\delta_\eps+W^{\delta_\eps}_{s\wedge m_\cJ}+S^{\delta_\eps}_{s\wedge m_\cJ}=:Z^\eps_{s\wedge m_\cJ}
\]
with $S^{\delta_\eps}_{s\wedge m_\cJ}=R^{\delta_\eps}_{A_s\wedge\tau_\cJ}$ and $m_\cJ=m^\eps_\cJ$ the first time the process $Z^\eps$ leaves the interval $(a,x_1)$ (let us also recall that the Brownian motion $W^{\delta_\eps}$ depends on the initial point $x_0+\delta_\eps$). By construction and recalling \eqref{eq:SK} we have that the process $Z^\eps$ solves (uniquely) the classical Skorokhod reflection problem
\begin{align}\label{eq:SK2}
&Z_{t\wedge m_\cJ}^{\varepsilon} \ge x_0+\delta_{\varepsilon},\quad\text{for all $t\ge 0$ and}\quad\ud S^{\delta_\eps}_t= \ind_{\{Z^\eps_t=x_0+\delta\}}\ud S^{\delta_\eps}_t\,.
\end{align}
Therefore we have an explicit formula for the increasing process $S^{\delta_\eps}$ (see, \cite[Lemma 6.14, Chapter 3]{KS}):  
\begin{align*}\label{eq:S}
S^{\delta_\eps}_t=\sup_{0\le s\le t}\big(- W^{\delta_\eps}_s\big).
\end{align*}
It may be worth noticing that reversing this construction gives another proof of the existence and uniqueness of the solution of the original reflected SDE for $\XX^\eps$. 

From \eqref{eq:mJ} we have
\[
\tau^\eps_1<t_1-t\iff m^\eps_\cJ<\int_0^{t_1-t}\sigma^2(\XX^\eps_s)\ud s\Longleftarrow m_\cJ^\eps<\underline{\sigma}^2(t_1-t),
\]
where $\underline \sigma:=\min_{x\in\R}\sigma(x)$ (recall that we extended $\sigma$ to $\R$ so that it is also strictly separated from zero). Hence
\begin{align}\label{prob}
\P(\tau^\eps_1<t_1-t)\ge \P\big(m^\eps_\cJ<\underline{\sigma}^2(t_1-t)\big).
\end{align}
As in the proof of Proposition \ref{prop:reg-b} (see \eqref{eq:law}) we need to pass to auxiliary processes 
\[
\widetilde{Z}^\eps_t=x_0+\delta_\eps+B_t+\widetilde{S}_t\quad \text{and}\quad \widetilde{S}_t=\sup_{0\le s\le t}(-B_s)
\]
in order to remove the dependence of the Brownian motion on the initial point. Then, setting $\widetilde{m}^\eps_\cJ=\inf\{s\ge 0:\widetilde{Z}^\eps_s\notin(a, x_1)\}$
and recalling $m^\eps_\cJ=\inf\{s\ge 0:Z^\eps_s\notin(a, x_1)\}$, we have 
\[
\P\big(m^\eps_\cJ<\underline{\sigma}^2(t_1-t)\big)=\P\big(\widetilde{m}^\eps_\cJ<\underline{\sigma}^2(t_1-t)\big),
\]
since $\mathsf{Law}_{\P}(Z^\eps,m^\eps_\cJ)=\mathsf{Law}_{\P}(\widetilde{Z}^\eps,\widetilde{m}^\eps_\cJ)$.

It is immediate to check that
\[
\{\widetilde{m}^\eps_\cJ< \underline{\sigma}^2(t_1-t)\} \downarrow  \{\widetilde{m}^0_\cJ<\underline{\sigma}^2(t_1-t)\}\quad\text{as $\eps\to 0$},
\]
with $\widetilde{m}^0_\cJ=\inf\{s\ge 0:\widetilde{Z}^0_s\notin(a, x_1)\}$ and where $\widetilde{Z}^0_s=x_0+B_s+S_s$ is a Brownian motion reflecting at $x_0$.
Hence, from \eqref{prob} and the above construction we have
\begin{align*}
\liminf_{\eps \downarrow 0} \P(\tau^\eps_1<t_1-t)\ge \P\big(\widetilde{m}^0_\cJ<\underline{\sigma}^2(t_1-t)\big)\ge \P\big(|B_{\underline{\sigma}^2(t_1-t)}|> (x_1-x_0)\big)>0,
\end{align*}
where in the second inequality we used that
\[
\{\widetilde{m}^0_\cJ<\underline{\sigma}^2(t_1-t)\}\supset \{\widetilde{Z}^0_{\underline{\sigma}^2(t_1-t)}> x_1\}
\]
and that, for each $s\in[0,T]$, the law of $\widetilde{Z}^0_s$ is the same as the law of $x_0+|B_s|$ (see, e.g., \cite[Thm.\ 6.17, Sec.\ 3.6.C]{KS}).

Finally, using Fatou's lemma in \eqref{eq:posit}, and the discussion above, we conclude
\begin{align*}
0\ge \liminf_{\eps \downarrow 0} \int_{t_0}^{t_1}\P(\tau^\eps_1<t_1-t)\varphi(t) \ud t \ge \int_{t_0}^{t_1} \P\big(|B_{\underline{\sigma}^2(t_1-t)}|> (x_1-x_0)\big)\varphi(t) \ud t>0,
\end{align*}
where the final inequality uses that $\varphi\ge 0$ and arbitrary. Hence a contradiction and continuity of $x\mapsto c(x)$ is proved.
\end{proof}

Thanks to Corollary \ref{cor:monot} we know that the boundary $c$ admits a continuous inverse on $[a,a_*)$ and on $(b_*,b]$. In particular, denoting $\bar c=c(a_*)=c(b_*)$ and setting $c_1(x)=c(x)$ for $x\in[a,a_*)$ and $c_2(x)=c(x)$ for $x\in(b_*,b]$ we can define 
\begin{align}\label{eq:b1b2}
b_1(t):=c^{-1}_1(t)\quad\text{and}\quad b_2(t):=c^{-1}_2(t)\quad\text{for $t\in[\bar c, T)$}.
\end{align}
Notice that if $c(a)<T$ (respectively $c(b)<T$) we understand $b_1$ (respectively $b_2$) to be constant and equal to $a$ for $t\in[c(a),T]$ (respectively equal to $b$ for $t\in[c(b),T]$).
Combining Theorem \ref{thm:non-flat} and Corollary \ref{cor:monot} we obtain the next result.  
\begin{corollary}\label{cor:non-flat}
Let Assumptions \ref{ass:cont-v} and \ref{ass:const} hold. Let $(a,b)\subset\Lambda^0_-$ and $\sigma, r, g\in C^{1,\alpha}([a,b])$ for some $\alpha\in(0,1)$. Recall the boundaries $b_1$ and $b_2$ defined in \eqref{eq:b1b2}. Then
\begin{align*}
&\text{the map $t\mapsto b_1(t)$ is {\em strictly decreasing} on $[\bar c, c(a))$,}\\
&\text{the map $t\mapsto b_2(t)$ is {\em strictly increasing} on $[\bar c, c(b))$.}
\end{align*}  
\end{corollary}
This result immediately applies to the setting of Example \ref{ex:bays}. Moreover, as a by-product we obtain the first known probabilistic proof of the strict monotonicity of the American put boundary.
\begin{example}[American put boundary]\label{ex:non-flat-b}
{\rm
Let us consider the classical Black and Scholes set-up where
\[
\ud Y_t=rY_t\,\ud t+\sigma Y_t\, \ud B_t,\qquad Y_0=y,
\] 
is the stock's dynamics with $r,\sigma>0$. Let $K>0$ be the strike price and $(x)^+:=\max\{0,x\}$, then the value of the American put option is
\[
\tilde v(t,y)=\sup_{0\le \tau\le T-t}\E_y\left[\e^{-r\tau}(K-Y_\tau)^+\right].
\]
Although this problem is perhaps the best studied optimal stopping problem in the literature, it is convenient to rewrite some of the main results in the notation of our work so far. The scale function of the process (up to affine transformations) reads $S(y)= (1-D)^{-1}y^{1-D}$ with $D=2r/\sigma^2$. Recalling the argument from Remark \ref{rem:process} we set $X_t=S(Y_t)$ and find the dynamics
\[
\ud X_t=(1-D)\sigma X_t\,\ud B_t,\qquad X_0=x=(1-D)^{-1}y^{1-D}.
\]
It is worth noticing that if $D>1$ the process $X$ is strictly negative, while if $D<1$ then the process $X$ is positive. For simplicity but with no loss of generality let us consider $D<1$. 

Setting $v(t,x)=(1-D)^{-1/(1-D)}\tilde v(t,S^{-1}(x))$ and $K'=K(1-D)^{-1/(1-D)}$, the optimal stopping problem becomes
\[
v(t,x)=\sup_{0\le \tau\le T-t}\E_x\left[\e^{-r\tau}\Big(K'-(X_\tau)^{\frac{1}{1-D}}\Big)^+\right].
\]
We now set $g(x)=(K'-(x)^{\frac{1}{1-D}})^+$ and notice that $g''(\ud x)$ has a positive atom at $\bar K=(K')^{1-D}$ with $g''(\{\bar K\})=(1-D)^{-1}\bar K^{D/(1-D)}$. Then, using \eqref{eq:u} (see also Remark \ref{rem:gain}) we obtain
\[
u(t,x)=\sup_{0\le \tau\le T-t}\E_x\left[\tfrac{1}{2}\int_{\R_+}\ell^z_\tau \mu(\ud z)\right],
\]
where 
\[
\mu(\ud z)=g''(\{\bar K\})\delta_{\bar K}(\ud z)-\left(\tfrac{D}{(1-D)^{2}}z^{-1/(1-D)}+Dz^{-2}g(z)\right)\ind_{\{z<\bar K\}}\ud z.
\]
Here we have $\Lambda_-=\R_+\setminus\{\bar K\}$, $\Lambda^0_-=(0,\bar K)$ and $\Lambda_+=\{\bar K\}=\Lambda^0_+$, which is a similar situation to the one in Example \ref{ex:bays}. Intuitively, waiting is costly for the option holder at all times $t\in[0,T]$ for which $X_t< \bar K$, whereas waiting is rewarding at times $t\in[0,T]$ when $X_t=\bar K$ (the option holder receives a reward at the `rate' of $\tfrac{1}{2}g''(\{\bar K\})\ud \ell^{\bar K}_t$). Differently to Example \ref{ex:bays}, here $(\bar K,\infty)=\Lambda_-\setminus\Lambda^0_-$ so that the option holder incurs no costs and no benefits when waiting if $X_t\in (\bar K,\infty)$. 

From (i) in Proposition \ref{prop:stop2} we obtain that $c(x)<T$ for all $x\in(0,\bar K)$, whereas Proposition \ref{prop:bay} gives $c(\bar K)=T$. In addition one can easily prove $v(t,x)>0$ for $t<T$ and all $x\in\R_+$. Then $c(x)=T$ for $x\in(\bar K,\infty)$ as well. By analogous arguments to those in the third paragraph of Example \ref{ex:bays} for any $x\in(0,\bar K)$ we have $[c(x),T]\times [0,x]\in\cS$. Finally, Corollary \ref{cor:monot} implies that  $c$ is strictly increasing on $(0,\bar K)$, hence it can be inverted defining a continuous, non-decreasing boundary $t\mapsto b(t)$. In the original coordinates $(t,y)$ the optimal exercise boundary reads $\hat b(t)=S^{-1}(b(t))=(1-D)^{1/(1-D)}(b(t))^{1/(1-D)}$. The latter is the familiar parametrisation of the American put exercise boundary (see, e.g., \cite[Ch.\ VII, Sec.\ 25.2]{PeSh}). Now, applying Corollary \ref{cor:non-flat} with $\Lambda^0_-=(0,\bar K)$ and $\Sigma_{[0,\bar K]}=[0,b(0)]$ we conclude that $t\mapsto b(t)$ must be strictly increasing. Hence $t\mapsto \hat b(t)$ is strictly increasing too.}
\end{example}

\section*{Appendix}
\subsection*{Proof of \eqref{eq:u}}
The process $X$ is a continuous local martingale, so It\^o-Tanaka-Meyer formula (\cite[Thm.\ VI.1.5]{RY}) gives:
\begin{align}\label{eq:ITg0}
g(X_t)=g(x)+\int_0^t g'(X_s)\sigma(X_s)\ud B_s+\tfrac{1}{2}\int_\R L^z_t g''(\ud z).
\end{align}
Here the fact that $X$ is bound to evolve in $\cI=(\underline x, \overline x)$ implies that 
\begin{align}\label{eq:ITg00}
\int_\R L^z_t g''(\ud z)=\int_{\cI} L^z_t g''(\ud z),
\end{align}
since $L^z_t= 0$ for $z\notin \cI$. It is also worth recalling that $(t,z)\mapsto L^z_t$ can be chosen $\P$-a.s.\ continuous by \cite[Thm.\ VI.1.7 and Corollary VI.1.8]{RY}.
Then $g(X_t)$ is a continuous semi-martingale and by It\^o's product rule and \eqref{eq:ITg0} we have
\begin{align}\label{eq:ITg1}
&\e^{-\int_0^t r(X_s)\ud s}g(X_t)\notag\\
&=g(x)-\int_0^t \e^{-\int_0^s r(X_u)\ud u}r(X_s)g(X_s)\ud s+\int_0^t \e^{-\int_0^s r(X_u)\ud u}\ud g(X_s)\\
&= g(x)-\int_0^t \e^{-\int_0^s r(X_u)\ud u}r(X_s)g(X_s)\ud s+\int_0^t\e^{-\int_0^s r(X_u)\ud u}g'(X_s)\sigma(X_s)\ud B_s\notag\\
&\quad+\tfrac{1}{2}\int_\cI \int_0^t \e^{-\int_0^s r(X_u)\ud u}\ud L^z_s g''(\ud z),\notag
\end{align}
where we used Fubini's theorem to swap the order of integrals in the final term (this can be justified more formally using the same arguments as in \eqref{eq:ITg2} below, so we avoid repetitions here).

For the integral with respect to $\ud s$ we use the occupation times formula. Let us rewrite $g$ in terms of its positive and negative part, i.e., $g=[g]^+-[g]^-$. Since 
\[
z\mapsto \Phi^\pm(z):=\frac{[g(z)]^\pm r(z)}{\sigma^{2}(z)}
\] 
is a positive Borel-measurable function, then \cite[Corollary VI.1.6]{RY} holds and
\begin{align*}
Q_t^{\pm}:=\int_0^t r(X_s)\big[g(X_s)\big]^\pm\ud s=\int_0^t \Phi^\pm(X_s)\ud \langle X\rangle_s=\int_\cI \Phi^\pm(z)L^z_t\ud z,
\end{align*}
for all $t\in[0,T]$, $\P$-a.s. Notice that even though $\sigma(z)$ may vanish when $z$ approaches the endpoints of $\cI$, the final integral above is well-defined because the initial expression for $Q^\pm_t$ is always finite. For a.e.\ $\omega$ the mapping 
\[
t\mapsto\int_0^t r\big(X_s(\omega)\big)g\big(X_s(\omega)\big)\ud s=(Q^{+}_t(\omega)-Q^-_t(\omega))=Q_t(\omega)
\] 
defines a (finite) signed measure on $[0,T]$. Moreover, for simple functions $f:[0,T]\to\R_+$ (possibly depending on $\omega$ as well) it is easy to check that
\[
\int_0^T f(t)\ud Q_t(\omega)=\int_\cI \Big(\int_0^T f(t)\ud L^z_t(\omega)\Big)\sigma^{-2}(z)r(z)g(z)\ud z.
\]
Thus, by dominated convergence the equality extends to any bounded measurable function.
In particular, choosing $f(s)=\e^{-\int_0^s r(X_u)\ud u}\mathbf{1}_{[0,t]}(s)$ we deduce
\begin{align}\label{eq:ITg2}
&\int_0^t \e^{-\int_0^s r(X_u)\ud u}r(X_s)g(X_s)\ud s\\
&=\int_0^T f(s) \ud Q_s
=\int_\cI \Big(\int_0^t \e^{-\int_0^s r(X_u)\ud u} \ud L^z_s \Big)\sigma^{-2}(z)r(z)g(z)\ud z.\notag
\end{align}
Almost sure finiteness of the last integral is guaranteed by finiteness of the initial expression in the equation. See also \cite[Theorem 45.4, Chapter IV]{RW} and \cite[Exercise 1.15, Chapter VI.1]{RY}.

Combining \eqref{eq:ITg1} and \eqref{eq:ITg2}, for any stopping time $\tau\in[0,T]$ we obtain 
\begin{align}\label{eq:ITg3}
\e^{-\int_0^\tau r(X_s)\ud s}g(X_\tau)= g(x)+\tfrac{1}{2}\int_\cI \ell^z_\tau \mu(\ud z)+\int_0^\tau\e^{-\int_0^s r(X_u)\ud u}g'(X_s)\sigma(X_s)\ud B_s,
\end{align}
with $\ell^z_s$ and $\mu(\ud z)$ as in \eqref{eq:ell}. 
Let $(\tau_n)_{n\ge 1}$ be a localising sequence for the local martingale in \eqref{eq:ITg3}. Taking expectations we have
\[
\E\Big[\e^{-\int_0^{\tau\wedge\tau_n} r(X_s)\ud s}g(X_{\tau\wedge\tau_n})\Big]= g(x)+\E\Big[\tfrac{1}{2}\int_\cI \ell^z_{\tau\wedge\tau_n} \mu(\ud z)\Big].
\] 
Letting $n\to\infty$ we have $\tau\wedge\tau_n\uparrow \tau$ and 
\[
\lim_{n\to\infty}\E\Big[\e^{-\int_0^{\tau\wedge\tau_n} r(X_s)\ud s}g(X_{\tau\wedge\tau_n})\Big]=\E\Big[\e^{-\int_0^{\tau} r(X_s)\ud s}g(X_{\tau})\Big],
\]
by dominated convergence and \eqref{eq:integr-g}. By monotone convergence we also obtain
\[
\lim_{n\to\infty}\E\Big[\tfrac{1}{2}\int_\cI \ell^z_{\tau\wedge\tau_n} \mu^\pm(\ud z)\Big]=\E\Big[\tfrac{1}{2}\int_\cI \ell^z_{\tau} \mu^\pm(\ud z)\Big].
\]
Since Assumption \ref{ass:ell} holds, the above implies
\[
\lim_{n\to\infty}\E\Big[\tfrac{1}{2}\int_\cI \ell^z_{\tau\wedge\tau_n} \mu(\ud z)\Big]=\E\Big[\tfrac{1}{2}\int_\cI \ell^z_{\tau} \mu(\ud z)\Big].
\]
Then \eqref{eq:u} follows by combining the limits above.

\subsection*{Proof of Lemma \ref{lem:contBM}}
For simplicity and with no loss of generality we adopt the standard convention that $t\mapsto W_t(\omega)$ is continuous for all $\omega\in\Omega$. In what follows we denote 
\[
\tau^x_{[a,b]}=\inf\{t\ge 0: x+W_t\notin[a,b]\}
\] 
the exit time from the closed interval $[a,b]$. 
Fix $x_0\in[a,b]$ and let $(x_n)_{n\ge 1}\subset(a,b)$ be a sequence converging to $x_0$.
\vspace{+3pt}

{\bf Step 1}. 
Here we show that 
\begin{equation}\label{eq:lsc}
\liminf_{n\to\infty}\tau^{x_n}_{(a,b)}(\omega)\ge \tau^{x_0}_{(a,b)}(\omega),\quad\text{for all $\omega\in\Omega$.}
\end{equation}
If $x_0\in\{a,b\}$, then $\tau^{x_0}_{(a,b)}(\omega)=0$ for all $\omega\in\Omega$ simply by definition. Hence lower semi-continuity is obvious. If $x_0\in(a,b)$, then $\tau^{x_0}_{(a,b)}(\omega)>0$ for all $\omega\in\Omega$ by continuity of Brownian paths. Fix $\omega\in\Omega$ and take an arbitrary $\delta>0$ such that $\tau^{x_0}_{(a,b)}(\omega)>\delta$. Then there exists $\eps=\eps_{\delta,\omega}>0$ such that $x_0+W_t(\omega)\in[a+\eps,b-\eps]$ for all $t\in[0,\delta]$ and consequently $x_n+W_t(\omega)\in[a+\eps/2,b-\eps/2]$ for all $t\in[0,\delta]$ and for all $n$'s such that $|x_n-x_0|<\eps/2$. Therefore, $\liminf_{n\to\infty}\tau_{(a,b)}^{x_n}(\omega)\ge \delta$ and, since $\delta>0$ was arbitrary, we have 
\[
\liminf_{n\to\infty}\tau_{(a,b)}^{x_n}(\omega)\ge \tau_{(a,b)}^{x_0}(\omega).
\]  
Recalling that $\omega\in\Omega$ was also arbitrary, lower semi-continuity holds as in \eqref{eq:lsc}.
\vspace{+3pt}

{\bf Step 2}. 
Here we show that 
\begin{equation}\label{eq:usc}
\limsup_{n\to\infty}\tau_{[a,b]}^{x_n}(\omega)\le \tau_{[a,b]}^{x_0}(\omega),\quad\text{for all $\omega\in \Omega$}.
\end{equation}
Fix $\omega\in\Omega$ and take an arbitrary $\delta>0$ such that $\tau^{x_0}_{[a,b]}(\omega)<\delta$. Then there exists $t\in(\tau^{x_0}_{[a,b]}(\omega),\delta)$ such that $x_0+W_t(\omega)\notin [a,b]$. Since $\R\setminus[a,b]$ is an open set, we can find $n_0\ge 1$ sufficiently large that $x_n+W_t(\omega)\notin [a,b]$ for all $n\ge n_0$. Hence 
\[
\limsup_{n\to\infty}\tau^{x_n}_{[a,b]}(\omega)\le t<\delta.
\]  
Since $\delta>0$ was arbitrary
\[
\limsup_{n\to\infty}\tau_{[a,b]}^{x_n}(\omega)\le \tau_{[a,b]}^{x_0}(\omega),
\]
and recalling that $\omega\in\Omega$ was also arbitrary, upper semi-continuity holds as in \eqref{eq:usc}.
\vspace{+3pt}

{\bf Step 3}. Here we conclude the proof by showing that 
\begin{equation}\label{eq:ctau}
\tau^{x_0}_{(a,b)}=\lim_{n\to\infty}\tau^{x_n}_{(a,b)}=\lim_{n\to\infty}\tau^{x_n}_{[a,b]}=\tau^{x_0}_{[a,b]},\quad\P-\text{a.s.}
\end{equation}
For each $n\ge 0$ let $\Omega_{x_n}:=\{\omega\in\Omega:\tau^{x_n}_{(a,b)}(\omega)=\tau^{x_n}_{[a,b]}(\omega)\}$ (notice that we include $\Omega_{x_0}$). Then it follows from the strong Markov property for Brownian motion and the law of iterated logarithm that $\P(\Omega_{x_n})=1$ for each $n\ge 0$. Setting $\bar\Omega=\cap_{n\ge 0}\Omega_{x_n}$ we have $\P(\bar \Omega)=1$ and 
\begin{equation}\label{eq:obar}
\tau^{x_n}_{(a,b)}(\omega)=\tau^{x_n}_{[a,b]}(\omega)\quad\text{for all $n\ge 0$ and for all $\omega\in\bar \Omega$}.
\end{equation}
Thus, combining \eqref{eq:lsc} and \eqref{eq:usc} with \eqref{eq:obar} we have
\begin{align*}
\tau^{x_0}_{(a,b)}(\omega)\le& \liminf_{n\to\infty}\tau^{x_n}_{(a,b)}(\omega)=\liminf_{n\to\infty}\tau^{x_n}_{[a,b]}(\omega)
\le\limsup_{n\to\infty}\tau^{x_n}_{[a,b]}(\omega)\le\tau^{x_0}_{[a,b]}(\omega)=\tau^{x_0}_{(a,b)}(\omega)  
\end{align*}
for all $\omega\in\bar\Omega$. This proves \eqref{eq:ctau}.\hfill$\square$
\vspace{+5pt}

\subsection*{Proof of \eqref{eq:ellprob}}
First of all we notice that since $X_{t\wedge\tau_\eps}\in\overline{\cI_\eps}$ for all $t\ge 0$, then 
\[
\E_{x_0}[\ell^z_{\tau_\eps}]>0\iff \E_{x_0}[L^z_{\tau_\eps}]>0,
\]
because the discount factor is bounded from below by $\e^{-\bar r_\eps T}$ with $\bar r_\eps =\sup_{x\in\cI_\eps}r(x)$.

Now we can use the same time-change as the one adopted in step 2 of the proof of Proposition \ref{prop:reg-b} (see \eqref{eq:mJ}--\eqref{eq:XZ}) with $\tau_\cJ$ therein replaced by $\tau_\eps$. Thus we get
\[
X^{x_0}_{s\wedge \tau_\eps}=Z^{x_0}_{\langle M\rangle_{s\wedge\tau_\eps}}=x_0+W_{\langle M\rangle_{s\wedge\tau_\eps}},
\]
where $W=W^{x_0}$ depends on $x_0$ but we can drop this dependence from our notation as $x_0$ is fixed throughout the proof. Let us denote by $(\widetilde L^z_{t})_{t\ge 0}$ the local time of the process $Z^{x_0}$ at $z\in\cI_{\eps}$. From It\^o-Tanaka's formula we get
\begin{align*}
L^z_{\tau_\eps}=\,&\big|X^{x_0}_{\tau_\eps}-z\big|-|x_0-z|-\int_0^{\tau_\eps}\mathrm{sign}(X^{x_0}_s-z)\sigma(X^{x_0}_s)\ud B_s\\
=\,&\big|Z^{x_0}_{\langle M\rangle_{\tau_\eps}}-z\big|-|x_0-z|-\int_0^{\langle M\rangle_{\tau_\eps}}\mathrm{sign}(Z^{x_0}_s-z)\ud W_s=\widetilde L^z_{\langle M\rangle_{\tau_\eps}},
\end{align*}
where the second equality is by \cite[Prop.\ 3.4.8]{KS} and the final one is by It\^o-Tanaka's formula applied to $|Z^{x_0}-z|$. So our problem reduces to proving that 
\[
\E_{x_0}\big[\widetilde L^z_{\langle M\rangle_{\tau_\eps}}\big]>0,\quad\text{for all $z\in\cI_\eps$}.
\]

Recall that $\langle M\rangle_{\tau_\eps}=m_{\eps}\wedge \langle M\rangle_{T-t}$ with $m_\eps =\inf\{t\ge 0: Z_t\notin \cI_\eps\}$. Moreover, setting $\underline \sigma_\eps=\inf_{x\in\cI_\eps}\sigma(x)$, Assumption \ref{ass:sigma1} gives $m_\eps\wedge\langle M\rangle_{T-t}\ge m_\eps\wedge[\underline{\sigma}^2_\eps(T-t)]$. Then
\begin{align*}
\E_{x_0}\big[L^z_{\tau_\eps}\big]=\E_{x_0}\big[\widetilde L^z_{\langle M\rangle_{\tau_\eps}}\big]\ge\E_{x_0}\big[\widetilde L^z_{m_\eps\wedge[\underline \sigma^2_\eps(T-t)]}\big].
\end{align*}
Set $t_\eps=\underline \sigma^2_\eps(T-t)$ for simplicity. Let $(\varphi_n)_{n\in\N}$ be smooth approximations of $\varphi(x):=|x|$ such that $\varphi_n\to \varphi$ uniformly on $\R$ with $\varphi'_n(x)\to \ind_{\{x\ge 0\}}-\ind_{\{x<0\}}$ pointwise and $\varphi''_n(x)\to 2\delta_0(x)$ in the sense of distributions. Then, taking expectations in It\^o-Tanaka's formula and using dominated convergence yields
\begin{align*}
\E_{x_0}\big[\widetilde L^z_{m_\eps\wedge t_\eps}\big]=&\lim_{n\to\infty}\E_{x_0}\big[\varphi_n(Z_{m_\eps\wedge t_\eps}-z)-\varphi_n(x_0-z)\big]\\
=&\lim_{n\to\infty}\E_{x_0}\Big[\tfrac{1}{2}\int_0^{t_\eps}\ind_{\{s<m_\eps\}}\varphi''_n(Z_s-z)\ud s\Big]\\
=&\lim_{n\to\infty}\tfrac{1}{2}\int_0^{t_\eps}\int_{\cI_\eps} \varphi''_n(y-z)p_{\cI_\eps}(s,x_0,y)\ud y\,\ud s,
\end{align*}
where we used Fubini's theorem in the final line and $p_{\cI_\eps}(\cdot,\cdot,\cdot)$ is the transition density of the Brownian motion killed upon leaving the interval $\cI_\eps$ (cf.\ \cite[p.\ 180, Eq.\ 1.15.8]{BS}). Since $\partial_y p_{\cI_\eps}(s,x_0,\cdot)$ is continuous, then it is not difficult to check that
\[
\lim_{n\to\infty}\tfrac{1}{2}\int_0^{t_\eps}\int_{\cI_\eps} \varphi''_n(y-z)p_{\cI_\eps}(s,x_0,y)\ud y\ud s=\int_0^{t_\eps}p_{\cI_\eps}(s,x_0,z)\ud s>0,
\]
using integration by parts. That concludes the proof of \eqref{eq:ellprob}.
\hfill$\square$

\vspace{+5pt}

\subsection*{Proof of Lemma \ref{lem:ut}}
This proof repeats verbatim an argument from \cite[Lemma 4.7]{CDeAS20} but adapted to out notation and setting.
By contradiction we assume there is $(t_0,x_0)\in\cC$ with $x_0\in(a,b)$ such that $\partial_t v(t_0,x_0)=0$. Since $v(t_0,x_0)>g(x_0)$ and $v(T,x_0)=g(x_0)$, there must exist $t_1\in(t_0,T)$ such that $(t_1,x_0)\in \cC$ and $\partial_t v(t_1,x_0)<-\eps$, for some $\eps>0$. By continuity of $\partial_t v$ inside $\cC\cap ([0,T)\times(a,b))$ (recall Remark \ref{rem:loc-reg}), and the fact that the set is open, there exists $\delta>0$ such that $\partial_t v(t_1,x)<-\eps/2$ for $x\in(x_0-\delta,x_0+\delta)\subset(a,b)$.

Letting $\cO:=(t_0,t_1)\times(x_0-\delta,x_0+\delta)$, we have that $\partial_t v\in C^{1,2}(\cO)$ thanks to internal regularity results for solutions of partial differential equations applied to \eqref{eq:PDEv} (see, e.g., \cite[Thm.~10, Ch.~3, Sec.~5]{Friedman}). Moreover, differentiating \eqref{eq:PDEv} with respect to time and recalling that $t\mapsto v(t,x)$ is non-increasing, we obtain that $\hat v:= \partial_t v$ solves
\begin{align}
&\big(\partial_t\hat v+(\sigma^2/2)\partial_{xx} \hat v\big)(t,x) = r(x)\hat v(t,x),\qquad\text{for $(t,x)\in\cO$},\\
&\hat v(t,x_0\pm\delta)\le 0,\qquad\qquad\qquad\qquad\qquad\quad\:\:\,\text{for $t\in[t_0,t_1)$},\\
&\hat v(t_1,x)<-\eps/2,\qquad\qquad\qquad\qquad\qquad\quad\:\:\,\text{for $x\in(x_0-\delta,x_0+\delta)$}.
\end{align}
Setting 
\[
\tau_\cO:=\inf\{s\ge 0\,:\,(t_0+s,X^{x_0}_s)\notin\cO\},
\]
an application of Dynkin's formula gives the following contradiction:
\begin{align*}
0=&\,\hat v(t_0,x_0)\\
=&\E_{x_0}\left[\e^{-\int_0^{\tau_\cO}r(X_s)\ud s}\hat v (t_0+\tau_{\cO},X_{\tau_\cO})\right]\le -\frac{\eps}{2}\E_{x_0}\left[\e^{-\int_0^{\tau_\cO}r(X_s)\ud s}\ind_{\{\tau_\cO=t_1-t_0\}}\right]<0,
\end{align*}
where the strict inequality holds because the process $(t_0+s,X^{x_0}_s)$ exits $\cO$ by crossing the segment $\{t_1\}\times(x_0-\delta,x_0+\delta)$ with positive probability, i.e., $\P_{x_0}(\tau_\cO=t_1-t_0)>0$.

\hfill$\square$
\vspace{+5pt}

\subsection*{Proof of Lemma \ref{lem:sm-f}}
For future reference let us denote $\cC_{a,b}=\cC\cap \big([0,T)\times(a,b)\big)$. Recall that $v\in C^{1,2}(\cC_{a,b})$ by Corollary \ref{cor:PDE} and Remark \ref{rem:loc-reg}. Then $\partial_x v$ is continuous separately in $\cC_{a,b}$ and in the interior of the stopping set $\mathrm{int}(\cS)\cap \big([0,T)\times(a,b)\big)$. Then we only need to look at the regularity across the boundary $\partial\cC_{a,b}$. An important observation which will be used several times below is that 
\begin{align}\label{eq:vxx}
\text{$\partial_{xx}v$ is continuous on $\overline{\cC\cap([0,T-\delta)\times(a,b))}$},
\end{align} 
for any $\delta>0$, thanks to Corollary \ref{cor:C1-space}.

With no loss of generality we assume $\hat c>0$ as the argument for $\hat c=0$ is analogous. In this case Corollary \ref{cor:monot} implies $\Sigma_{[a,b]}= a_*$. If $a_*=a$ or $a_*=b$ then the boundary $x\mapsto c(x)$ is strictly monotonic on $(a,b)$. In the more general situation when $a_*\in(a,b)$ we need to consider separately the intervals $(a,a_*]$ and $[a_*,b)$ where the boundary is strictly decreasing and strictly increasing, respectively. Below we develop our arguments only for $x\in[a_*,b)$ as the remaining case follows along the same lines up to obvious changes. 

Take $a<a'<a_*<b'<b$ and for any $x\in(a',b')$ let 
\begin{align}\label{eq:tau0}
\tau^x_{0}=\inf\{s\ge 0:X^x_s\notin(a',b')\}.
\end{align} 
Take $\widetilde\sigma\in C^1_b(\R)$ as an extension of $\sigma$ outside the interval $(a,b)$. Letting $\widetilde X$ be the unique strong solution of 
\[
\ud\widetilde X_t=\widetilde \sigma(\widetilde X_t)\ud B_t,\qquad\widetilde X_0=x,
\] 
and $\widetilde \tau_{0}^{\,x}$ the exit time of $\widetilde X^x$ from $(a',b')$ we have $\P$-a.s. the equalities
\begin{align}\label{eq:Pas}
\tau^x_0=\widetilde\tau_0^{\,x}\quad\text{and}\quad X^x_{s\wedge\tau_0}= \widetilde X^x_{s\wedge\widetilde \tau_{0}}\quad\text{for all $s\ge 0$.}
\end{align} 
We will use such equivalence later on.

Fix $x\in(a_*,b')$ with $(t,x)\in \cC$ and $t> \hat c$. Take $\eps>0$ such that $x+\eps<b'$ and let $\rho_\eps=\tau^{x}_{0}\wedge\tau^{x+\eps}_{0}$. Taking $\tau_*^{t,x}$ optimal for $v(t,x)$ and $\tau_*^{t,x+\eps}$ optimal for $v(t,x+\eps)$ we notice that $\tau_*^{t,x}\wedge\rho_\eps\le \tau_*^{t,x+\eps}\wedge\rho_\eps$, $\P$-a.s.\ because the boundary $x\mapsto c(x)$ is increasing on $[a_*,b)$ and, since $t>\hat c$, the process $(t+s,X^x_s)$ cannot enter the rectangle $(\hat c,T]\times(a,a_*]$ before hitting the stopping set (recall $[\hat c,T]\times\{a_*\}\in\cS$). Then, letting $\tau_*=\tau_*^{t,x}$ for simplicity and using the martingale property of the value (see \eqref{eq:martv}) we have
\begin{align*}
&v(t,x)=\E\left[\e^{-\int_0^{\tau_*\wedge\rho_\eps}r(X^x_s)\ud s}v(t\!+\!\tau_*\wedge\rho_\eps,X^x_{\tau_*\wedge\rho_\eps})\right]\quad\text{and}\\
&v(t,x\!+\!\eps)= \E\left[\e^{-\int_0^{\tau_*\wedge\rho_\eps}r(X^{x+\eps}_s)\ud s}v(t\!+\!\tau_*\wedge\rho_\eps,X^{x+\eps}_{\tau_*\wedge\rho_\eps})\right].
\end{align*}
Subtracting the two expressions we obtain
\begin{align}\label{eq:C1-0}
&v(t,x\!+\!\eps)-v(t,x)\\
&=\E\left[\left(\e^{-\int_0^{\tau_*\wedge\rho_\eps}r(X^{x+\eps}_s)\ud s}-\e^{-\int_0^{\tau_*\wedge\rho_\eps}r(X^{x}_s)\ud s}\right)v(t\!+\!\tau_*\wedge\rho_\eps,X^{x+\eps}_{\tau_*\wedge\rho_\eps})\right]\notag\\ 
&\quad+\E\left[\e^{-\int_0^{\tau_*\wedge\rho_\eps}r(X^{x}_s)\ud s}\Big(v(t\!+\!\tau_*\wedge\rho_\eps,X^{x+\eps}_{\tau_*\wedge\rho_\eps})-v(t\!+\!\tau_*\wedge\rho_\eps,X^x_{\tau_*\wedge\rho_\eps})\Big)\right].\notag
\end{align}

First we obtain a lower bound. For the first term in \eqref{eq:C1-0} we recall that $v$ is bounded on compacts (see \eqref{eq:bound-v}), we set $\Delta^\eps X_t=X^{x+\eps}_t-X^x_t$ and use the mean value theorem and $r\ge 0$ to obtain
\begin{align}\label{eq:intb0}
&\E\left[\left(\e^{-\int_0^{\tau_*\wedge\rho_\eps}r(X^{x+\eps}_s)\ud s}-\e^{-\int_0^{\tau_*\wedge\rho_\eps}r(X^{x}_s)\ud s}\right)v(t\!+\!\tau_*\wedge\rho_\eps,X^{x+\eps}_{\tau_*\wedge\rho_\eps})\right]\\
&\ge -\bar v\,\E\left[\left|\e^{-\int_0^{\tau_*\wedge\rho_\eps}r(X^{x+\eps}_s)\ud s}-\e^{-\int_0^{\tau_*\wedge\rho_\eps}r(X^{x}_s)\ud s}\right|\right]\notag\\
&\ge -\bar v\,\E\left[\int_0^{\tau_*\wedge\rho_\eps}\!\!\int^{\Delta^\eps X_s}_0\left|r'(X^{x}_s+z)\right|\ud z\ud s\right]\notag\\
&\ge -\bar v\,\bar r\,\E\left[\int_0^{\tau_*\wedge\rho_\eps}\Delta^\eps X_s\ud s\right]\notag
\end{align}
where 
\[
\bar v=\sup_{(s,z)\in[0,T]\times[a,b]}|v(s,z)|\quad\text{and}\quad \bar r=\sup_{z\in[a,b]}|r'(z)|.
\]
For the second term in \eqref{eq:C1-0}, recalling that $v(t+\tau_*,X^x_{\tau_*})=g(X^x_{\tau_*})$ and $v(t+\tau_*,X^{x+\eps}_{\tau_*})\ge g(X^{x+\eps}_{\tau_*})$ by optimality of $\tau_*=\tau_*^{t,x}$, we obtain
\begin{align*}
&\E\left[\e^{-\int_0^{\tau_*\wedge\rho_\eps}r(X^{x}_s)\ud s}\Big(v(t\!+\!\tau_*\wedge\rho_\eps,X^{x+\eps}_{\tau_*\wedge\rho_\eps})-v(t\!+\!\tau_*\wedge\rho_\eps,X^x_{\tau_*\wedge\rho_\eps})\Big)\right]\\
&\ge \E\left[\ind_{\{\tau_*\le \rho_\eps\}}\e^{-\int_0^{\tau_*\wedge\rho_\eps}r(X^{x}_s)\ud s}\Big(g(X^{x+\eps}_{\tau_*})-g(X^x_{\tau_*})\Big)\right]\notag\\
&\quad+\E\left[\ind_{\{\tau_*> \rho_\eps\}}\e^{-\int_0^{\tau_*\wedge\rho_\eps}r(X^{x}_s)\ud s}\Big(v(t\!+\!\rho_\eps,X^{x+\eps}_{\rho_\eps})-v(t\!+\!\rho_\eps,X^x_{\rho_\eps})\Big)\right].\notag
\end{align*}
Since $\tau_*\le T-t$ we have $\{\tau_*>\rho_\eps\}\subset\{\rho_\eps<T-t\}$. 
Since $x\mapsto c(x)$ is strictly increasing on $(a_*,b)$ with $c(x)<T$ (Proposition \ref{prop:stop2}), on the event $\{\tau_*>\rho_\eps\}$ it  holds
\[
b'=X^{x+\eps}_{\rho_\eps}\ge X^x_{\rho_\eps}> c^{-1}(t+\rho_\eps) 
\]
with $c^{-1}$ the continuous inverse of $c$ on $(a_*,b)$. Then, for any $b''\in(b',b)$, on the event $\{\tau_*>\rho_\eps\}$ the segment $\{t+\rho_\eps\}\times[X^x_{\rho_\eps},b'']$ lies in $\cC$ and we can use the fundamental theorem of calculus (twice) to obtain
\begin{align*}
&v(t\!+\!\rho_\eps,X^{x+\eps}_{\rho_\eps})-v(t\!+\!\rho_\eps,X^x_{\rho_\eps})\\
&=\int_0^{\Delta^\eps X_{\rho_\eps}}\partial_x v(t+\rho_\eps, X^x_{\rho_\eps}+z)\ud z\\
&=\int_0^{\Delta^\eps X_{\rho_\eps}}\Big(\partial_x v(t+\rho_\eps,b'')-\int_{X^x_{\rho_\eps}+z}^{b''}\partial_{xx}v(t+\rho_\eps,\zeta)\ud\zeta\Big)\ud z.
\end{align*}
Due to the strict monotonicity of the boundary $c$ and the fact that $c(x)<T$ for $x\in(a,b)$, there exists $\delta>0$ such that $c(b')< c(b'')\le c(b)-\delta<T-\delta$. Moreover, by definition of $\rho_\eps$, on the event $\{\tau_*>\rho_\eps\}$ we also have $t+\rho_\eps\le c(b')$. Then, recalling \eqref{eq:vxx}, on the event $\{\tau_*>\rho_\eps\}$ we have 
\[
\sup_{\nu\in[X^x_{\rho_\eps}, b'']}|\partial_{xx}v(t+\rho_\eps,\nu)|\le \kappa
\]
for some $\kappa>0$, independent of $\eps$. Hence,
\begin{align}\label{eq:boundC}
&\Big|\partial_x v(t+\rho_\eps,b'')-\int_{X^x_{\rho_\eps}+z}^{b''}\partial_{xx}v(t+\rho_\eps,\nu)\ud \nu\Big|\\
&\le \sup_{s\in[0,c(b')]}\big|\partial_x v(s,b'')\big|+\kappa(b-a)=:C\notag
\end{align}
for some deterministic constant $C>0$ independent of $\eps$, 
where we use that $\partial_x v(\,\cdot,b'')$ is bounded on $[0,c(b')]$ since $v\in C^{1,2}(\cC_{a,b})$ and $[0,c(b')]\times\{b''\}\subset\cC_{a,b}$. Then, substituting the estimates above back into \eqref{eq:C1-0} we have
\begin{align}\label{eq:C1-1}
&v(t,x\!+\!\eps)-v(t,x)\\
&\ge \E\left[\e^{-\int_0^{\tau_*\wedge\rho_\eps}r(X^x_s)\ud s}\Delta^\eps X_{\tau_*\wedge\rho_\eps}\Big(\ind_{\{\tau_*\le \rho_\eps\}}\inf_{\xi\in[0,\Delta^\eps X_{\tau_*}]}g'(X^x_{\tau_*}+\xi)- \ind_{\{\tau_*> \rho_\eps\}}C \Big)\right]\notag\\
&\quad-\bar v\,\bar r\,\E\left[\int_0^{\tau_*\wedge\rho_\eps}\Delta^\eps X_s\ud s\right].\notag
\end{align}

Thanks to \eqref{eq:Pas} and due to the local nature of the argument we are using, we may substitute $X$ with $\widetilde X$ in all our calculations above. Therefore, there is no loss of generality assuming that $x\mapsto X^x$ is continuously differentiable in all the expressions above (since $x\mapsto \widetilde X^x$ is such by, e.g., \cite[Ch. V.7]{Protter}) and moreover the process $t\mapsto \partial_x X^x_t$ evolves according to 
\[
\partial_x X^x_t=1+\int_0^t\partial_x\widetilde\sigma(X^x_s)\partial_x X^x_s\ud B_s.
\]
In particular, $(t,x)\mapsto \partial_x X^x$ admits a continuous modifications (which we use in the rest of the proof) and 
\[
\partial_x X^x_t=\exp\left(\int_0^t\partial_x\widetilde\sigma(X^x_s)\ud B_s-\tfrac{1}{2}\int_0^t\big[\partial_x\widetilde\sigma(X^x_s)\big]^2\ud s\right).
\]
Thanks to the arbitrariness of $\widetilde \sigma$ and the explicit formula for $\partial_x X^x_t$ we can also assume with no loss of generality that
\begin{align}\label{eq:supDX}
\E[Z]:=\E\left[\sup_{x\in[a,b]}\sup_{t\in[0,T]}\partial_x X^x_t\right]<\infty.
\end{align}

Dividing both sides of \eqref{eq:C1-1} by $\eps$ and rewriting 
\[
\ind_{\{\tau_*\le \rho_\eps\}}\inf_{\xi\in[0,\Delta^\eps X_{\tau_*}]}g'(X^x_{\tau_*}+\xi)=\inf_{\xi\in[0,\Delta^\eps X_{\tau_*\wedge\rho_\eps}]}g'(X^x_{\tau_*\wedge\rho_\eps}+\xi)\big(1-\ind_{\{\tau_*> \rho_\eps\}}\big)
\]
we obtain 
\begin{align*}
&\frac{v(t,x\!+\!\eps)-v(t,x)}{\eps}\\
&\ge \E\left[\e^{-\int_0^{\tau_*\wedge\rho_\eps}r(X^x_s)\ud s}\Big(\inf_{\xi\in[0,\Delta^\eps X_{\tau_*\wedge\rho_\eps}]}g'(X^x_{\tau_*\wedge\rho_\eps}+\xi)- \ind_{\{\tau_*> \rho_\eps\}}C'\Big)\frac{1}{\eps}\int_0^\eps \partial_x X^{x+\zeta}_{\tau_*\wedge\rho_\eps}\ud \zeta\right]\notag\\
&\quad-\bar v\,\bar r\,\E\left[\int_0^{\tau_*\wedge\rho_\eps}\Big(\frac{1}{\eps}\int_0^\eps \partial_x X^{x+\zeta}_s\ud \zeta\Big)\ud s\right],
\end{align*}
where $C'=C+\sup_{x\in[a,b]}|g'(x)|$.
By time-change arguments as in step 2 of the proof of Proposition \ref{prop:reg-b} we can reduce $X$ to a Brownian motion and then apply Lemma \ref{lem:contBM} to obtain convergence in probability of $\rho_\eps$ to $\tau^x_0$ (analogously to \eqref{eq:sigma-cont}). Then, taking limits as $\eps_n\to 0$, along a suitable sequence $(\eps_n)$, we have $\rho_{\eps_n}\to\tau^x_0$, $\P$-a.s. Moreover,
\[
\Delta^{\eps_n} X_{\tau}\to 0\quad\text{and}\quad\frac{1}{\eps_n}\int_0^{\eps_n}\partial_x X^{x+\zeta}_{\tau}\ud \zeta\to \partial_x X^{x}_{\tau},\quad\text{$\P$-a.s., as $\eps_n\to 0$,}
\]
for any stopping time $\tau\in[0,T]$. Since $x\mapsto c(x)$ is strictly increasing on $(a_*,b)$ and $X^{x+\eps'}_\cdot\le X^{x+\eps}_\cdot$ for any $\eps'<\eps$, we have $\tau_*\wedge\rho_{\eps_n}\uparrow \tau_*\wedge\tau_0$ as $\eps_n\to 0$ so that $\{\tau_*> \rho_{\eps_n}\}\subset\{\tau_*> \rho_{\eps_0}\}$ for all $\eps_n\in(0,\eps_0]$ and some $\eps_0>0$ (recall that $\tau_*=\tau_*^{t,x}$ and therefore, prior to being absorbed, the process $(t+s\wedge\tau_*,X^{x+\eps}_{s\wedge\tau_*})$ can only leave the rectangle $(\hat c,T)\times(a',b')$ by either hitting $(\hat c,T)\times\{b'\}$ or $\{T\}\times(a_*,b')$).

Then, by continuity of $g'$ on $[a,b]$ and dominated convergence (recall \eqref{eq:supDX}) we obtain the lower bound
\begin{align}\label{eq:lb}
\partial_x v(t,x)\ge& \E\left[\e^{-\int_0^{\tau_*\wedge\tau_0}r(X^x_s)\ud s}\Big(g'(X^x_{\tau_*\wedge\tau_0})-\ind_{\{\tau_*> \rho_{\eps_0}\}}C'\Big)\partial_x X^x_{\tau_*\wedge\tau_0}\right]\\
&-\bar v\,\bar r\,\E\left[\int_0^{\tau_*\wedge\tau_0}\partial_x X^x_s\ud s\right].\notag
\end{align}

For the upper bound, starting from \eqref{eq:C1-0} we have 
\begin{align*}
&\E\left[\left(\e^{-\int_0^{\tau_*\wedge\rho_\eps}r(X^{x+\eps}_s)\ud s}-\e^{-\int_0^{\tau_*\wedge\rho_\eps}r(X^{x}_s)\ud s}\right)v(t\!+\!\tau_*\wedge\rho_\eps,X^{x+\eps}_{\tau_*\wedge\rho_\eps})\right]\\
&\le \bar v\,\bar r\,\E\left[\int_0^{\tau_*\wedge\rho_\eps}\Delta^\eps X_s\ud s\right],\notag
\end{align*}
by the same argument as in \eqref{eq:intb0}. For the second term in \eqref{eq:C1-0} we have

\begin{align}\label{eq:C1-2}
&\E\left[\e^{-\int_0^{\tau_*\wedge\rho_\eps}r(X^{x}_s)\ud s}\Big(v(t\!+\!\tau_*\wedge\rho_\eps,X^{x+\eps}_{\tau_*\wedge\rho_\eps})-v(t\!+\!\tau_*\wedge\rho_\eps,X^x_{\tau_*\wedge\rho_\eps})\Big)\right]\\
& =\E\left[\ind_{\{\tau_*\le \rho_\eps\}}\e^{-\int_0^{\tau_*\wedge\rho_\eps}r(X^{x}_s)\ud s}\Big(v(t+\tau_*,X^{x+\eps}_{\tau_*})-g(X^x_{\tau_*})\Big)\right]\notag\\
&\quad+\E\left[\ind_{\{\tau_*> \rho_\eps\}}\e^{-\int_0^{\tau_*\wedge\rho_\eps}r(X^{x}_s)\ud s}\Big(v(t\!+\!\rho_\eps,X^{x+\eps}_{\rho_\eps})-v(t\!+\!\rho_\eps,X^x_{\rho_\eps})\Big)\right].\notag
\end{align}
The second term on the right-hand side can be treated with the same estimates as in \eqref{eq:boundC} and gives
\begin{align*}
&\E\left[\ind_{\{\tau_*> \rho_\eps\}}\e^{-\int_0^{\tau_*\wedge\rho_\eps}r(X^x_s)\ud s}\Big(v(t\!+\!\rho_\eps,X^{x+\eps}_{\rho_\eps})-v(t\!+\!\rho_\eps,X^x_{\rho_\eps})\Big)\right]\\
&\le C\,\E\left[\ind_{\{\tau_*> \rho_\eps\}}\e^{-\int_0^{\tau_*\wedge\rho_\eps}r(X^x_s)\ud s}\Delta^\eps X_{\tau_*\wedge\rho_\eps}\right].
\end{align*}
For the remaining term in \eqref{eq:C1-2} we notice that, on the event $\{\tau_*\le \rho_\eps\}$, strict monotonicity of the boundary implies $t+\tau_*\le c(b')<c(b'')\le c(b)-\delta$,  $\P$-a.s., with $b''\in(b',b)$. Therefore, arguing as in \eqref{eq:boundC} and using that $(t+\tau_*,X^{x+\eps}_{\tau_*}+n^{-1})\in\cC_{a,b}$ on the event $\{\tau_*\le \rho_\eps\}$ for all sufficiently large $n$'s, we get
\begin{align*}
&v(t+\tau_*,X^{x+\eps}_{\tau_*})\\
&=\lim_{n\to\infty}v(t+\tau_*,X^{x}_{\tau_*}+n^{-1}+\Delta^\eps X_{\tau_*})\\
&=\lim_{n\to\infty}\Big[v(t\!+\!\tau_*,X^{x}_{\tau_*}\!+\!n^{-1})+\int^{\Delta^\eps X_{\tau_*}}_{0}\partial_x v(t+\tau_*,X^x_{\tau_*}+n^{-1}+z)\ud z\Big]\\
&=\lim_{n\to\infty}\Big[v(t\!+\!\tau_*,X^{x}_{\tau_*}\!+\!n^{-1})\\
&\qquad\qquad+\!\int^{\Delta^\eps X_{\tau_*}}_{0}\!\!\Big(\partial_x v(t+\tau_*,b'')-\!\int^{b''}_{X^x_{\tau_*}+n^{-1}+z}\!\partial_{xx}v(t+\tau_*,\zeta)\ud\zeta\Big)\ud z\Big]\\
&= g(X^x_{\tau_*})+\int^{\Delta^\eps X_{\tau_*}}_{0}\Big(\partial_x v(t+\tau_*,b'')-\!\!\int^{b''}_{X^x_{\tau_*}+z}\!\partial_{xx}v(t+\tau_*,\zeta)\ud\zeta\Big)\ud z,
\end{align*}
where we are also using \eqref{eq:vxx} to justify the limit of the double integral. Combining the above we obtain
\begin{align*}
&\frac{v(t,x\!+\!\eps)-v(t,x)}{\eps}\\
&\le C\,\E\left[\ind_{\{\tau_*> \rho_\eps\}}\e^{-\int_0^{\tau_*\wedge\rho_\eps}r(X^x_s)\ud s}\Big(\frac{1}{\eps}\int_0^\eps \partial_x X^{x+\zeta}_{\tau_*\wedge \rho_\eps}\ud \zeta\Big)\right]\\
&\quad+\bar v\,\bar r\,\E\left[\int_0^{\tau_*\wedge\rho_\eps}\Big(\frac{1}{\eps}\int_0^\eps \partial_x X^{x+\zeta}_s\ud \zeta\Big)\ud s\right]\\
&\quad+ \E\bigg[\ind_{\{\tau_*\le \rho_\eps\}}\e^{-\int_0^{\tau_*\wedge\rho_\eps}r(X^x_s)\ud s}\frac{1}{\eps}\int^{\Delta^\eps X_{\tau_*}}_{0}\!\!\Big(\partial_x v(t+\tau_*,b'')-\!\!\int^{b''}_{X^x_{\tau_*}+z}\!\partial_{xx}v(t+\tau_*,\zeta)\ud\zeta\Big)\ud z\bigg].
\end{align*}
We need a slightly more refined estimate for the last term above. In particular, recalling \eqref{eq:boundC} and rearranging the indicator functions we have
\begin{align*}
&\E\bigg[\ind_{\{\tau_*\le \rho_\eps\}}\e^{-\int_0^{\tau_*\wedge\rho_\eps}r(X^x_s)\ud s}\frac{1}{\eps}\int^{\Delta^\eps X_{\tau_*}}_{0}\!\!\Big(\partial_x v(t+\tau_*,b'')-\!\!\int^{b''}_{X^x_{\tau_*}+z}\!\partial_{xx}v(t+\tau_*,\zeta)\ud\zeta\Big)\ud z\bigg]\\
&\le\E\bigg[\e^{-\int_0^{\tau_*\wedge\rho_\eps}r(X^x_s)\ud s}\cdot\\
&\qquad\qquad\cdot\frac{1}{\eps}\int^{\Delta^\eps X_{\tau_*\wedge\rho_\eps}}_{0}\Big(\partial_x v(t+\tau_*\wedge\rho_\eps,b'')-\!\!\int^{b''}_{X^x_{\tau_*\wedge\rho_\eps}+z}\!\partial_{xx}v(t+\tau_*\wedge\rho_\eps,\zeta)\ud\zeta\Big)\ud z\bigg]\\
&\quad+ C \E\left[\ind_{\{\tau_*> \rho_\eps\}}\frac{1}{\eps}\Delta^\eps X_{\tau_*\wedge\rho_\eps}\right]
\end{align*}

As in \eqref{eq:lb}, we take limits as $\eps_n\to 0$ along the same subsequence $(\eps_n)$. In order to use dominated convergence we recall \eqref{eq:supDX}. Moreover, we notice that
\begin{equation}\label{eq:cc}
\begin{split}
&(s,z)\mapsto \int^{b''}_{X^x_{\tau_*\wedge\tau_0\wedge s}+z}\!\partial_{xx}v(t+\tau_*\wedge\tau_0\wedge s,\zeta)\ud\zeta\\
&\text{and}\quad s\mapsto \partial_x v(t+\tau_*\wedge\tau_0\wedge s,b'')
\end{split}
\end{equation}
are $\P$-a.s.\ bounded and continuous thanks to \eqref{eq:vxx} and since $t+\tau_*\wedge\tau_0\le c(b')\le c(b'')-\delta$ for some $\delta>0$, by strict monotonicity of $c$, so that $[0,c(b')]\times \{b''\}\subset \cC_{a,b}$. Finally, recalling that $\tau_*\wedge\rho_{\eps_n}\uparrow \tau_*\wedge\tau_0$ and $\{\tau_*>\rho_{\eps_n}\}\subset \{\tau_*>\rho_{\eps_0}\}$ for a fixed $\eps_0>0$ and all $\eps_n\in(0,\eps_0]$, we find the upper bound
\begin{align}\label{eq:ub}
&\partial_x v(t,x)\\
&\le 2C\,\E\left[\ind_{\{\tau_*> \rho_{\eps_0}\}}\partial_x X^x_{\tau_*\wedge \tau_0}\right]+\bar v\,\bar r\,\E\left[\int_0^{\tau_*\wedge\tau_0}\partial_x X^x_s\ud s\right]\notag\\
&\quad+ \E\bigg[\e^{-\int_0^{\tau_*\wedge\tau_0}r(X^x_s)\ud s}\partial_x X^x_{\tau_*\wedge\tau_0}\cdot\notag\\
&\qquad\qquad\cdot\bigg(\partial_x v(t+\tau_*\wedge\tau_0,b'')-\int^{b''}_{X^x_{\tau_*\wedge\tau_0}}\!\partial_{xx}v(t+\tau_*\wedge\tau_0,\zeta)\ud\zeta\bigg)\bigg].\notag
\end{align}

It remains to take limits in \eqref{eq:lb} and \eqref{eq:ub} along an arbitrary sequence $(t_n,x_n)_{n\ge 1}\subset \cC\cap\big((\hat c,T)\times[a_*,b)\big)$ that converges to $(t_0,x_0)\in\partial\cC\cap\big((\hat c,T)\times[a_*,b)\big)$. 
Clearly there is no loss of generality in assuming $x_0\in(a',b')$ with $a'<b'$ as above. Arguing by contradiction, assume that there is one such sequence $(t_n,x_n)$ for which 
\begin{align}\label{eq:bad}
\lim_{n\to\infty}\partial_x v(t_n,x_n)\neq g'(x_0).
\end{align}
Thanks to Proposition \ref{prop:reg-b} we can extract a subsequence, which we denote again by $(t_n,x_n)$, such that $\tau^n_*:=\tau^{t_n,x_n}_*\to 0$, $\P$-a.s. By the same arguments (i.e., time-change and Lemma \ref{lem:contBM}) we can also show that $\tau^n_0:=\tau^{x_n}_0\to \tau^{x_0}_0$ and $\rho^n_{\eps_0}=\tau^{x_n}_0\wedge\tau^{x_n+\eps_0}_0\to \tau^{x_0}_0\wedge\tau^{x_0+\eps_0}_0=\rho^{x_0}_{\eps_0}$, $\P$-a.s.\ (possibly selecting further subsequences). Since $\P(\tau^{x_0}_0\wedge\tau^{x_0+\eps_0}_0>0)=1$, recalling \eqref{eq:supDX} we get
\begin{align}\label{eq:l0}
\lim_{n\to\infty}\E\left[\partial_x X^{x_n}_{\tau^n_*\wedge\rho^n_{\eps_0}}\ind_{\{\tau^n_*> \rho^n_{\eps_0}\}}\right]\le \E\left[Z\lim_{n\to\infty}\ind_{\{\tau^n_*> \rho^n_{\eps_0}\}}\right]=0
\end{align}
and
\begin{align}\label{eq:l1}
\lim_{n\to\infty}\E\left[\int_0^{\tau^n_*\wedge\tau^n_0}\partial_x X^{x_n}_s\ud s\right]\le\E\left[Z\lim_{n\to\infty}(\tau^n_*\wedge\tau^n_0)\right] =0
\end{align}
by dominated convergence. Moreover, $\partial_x X^{x_n}_{\tau^n_*\wedge\tau^n_0}\to \partial_x X^{x_0}_0=1$ and $\partial_x v$ is continuous at $(t_0,b'')\in\cC_{a,b}$ (recall $b'<b''<b$ so that $c(b')<c(b'')<c(b)$). Then, by dominated convergence and the continuity arguments as in \eqref{eq:cc} we also get
\begin{align}\label{eq:l2}
\lim_{n\to\infty}&\,\E\bigg[\e^{-\int_0^{\tau^n_*\wedge\tau^n_0}r(X^{x_n}_s)\ud s}\partial_x X^{x_n}_{\tau^n_*\wedge\tau^n_0}\cdot\\
&\qquad\cdot\bigg(\partial_x v(t_n+\tau^n_*\wedge\tau^n_0,b'')-\int^{b''}_{X^{x_n}_{\tau^n_*\wedge\tau^n_0}}\!\partial_{xx}v(t_n+\tau^n_*\wedge\tau^n_0,\zeta)\ud\zeta\bigg)\bigg]\notag\\
=&\,\partial_x v(t_0,b'')-\int^{b''}_{x_0}\!\partial_{xx}v(t_0,\zeta)\ud\zeta=\partial_x v(t_0,x_0+),\notag
\end{align}
where the final equality holds and the right-limit $\partial_x v(t_0,x_0+)$ exists because 
$\zeta\mapsto\partial_{xx}v(t_0,\zeta)$ is continuous on $[x_0,b'']$ (see \eqref{eq:vxx}). Finally, we also have
\begin{align}\label{eq:l3}
\lim_{n\to\infty}\E\left[\e^{-\int_0^{\tau^n_*\wedge\tau^n_0}r(X^{x_n}_s)\ud s}g'(X^{x_n}_{\tau^n_*\wedge \tau^n_0})\partial_x X^{x_n}_{\tau^n_*\wedge \tau^n_0}\right]=g'(x_0)
\end{align}
by dominated convergence and continuity of $g'$ at $x_0$. We claim that $\partial_x v(t_0,x_0+)=g'(x_0)$ so that combining the limits \eqref{eq:l0}--\eqref{eq:l3} with \eqref{eq:lb} and \eqref{eq:ub} we obtain
\[
g'(x_0)\le \lim_{n\to\infty}\partial_{x}v(t_n,x_n)\le \partial_x v(t_0,x_0+)=g'(x_0).
\]
That contradicts \eqref{eq:bad} since the limit must be the same along any subsequence.

It remains to justify that $\partial_x v(t_0,x_0+)=g'(x_0)$. From the first two inequalities above we have $\partial_x v(t_0,x_0+)\ge g'(x_0)$ so, arguing by contradiction, we assume $\theta:=\partial_x v(t_0,x_0+)- g'(x_0)>0$. Notice that $x_0=c^{-1}(t_0)$,  with $c^{-1}$ the continuous inverse of $c(\,\cdot\,)$ on $[a_*,b)$, and the mapping $x\mapsto\partial_{xx} v(t_0, x)$ is continuous in $[x_0,b)$. We must consider separately the case $x_0>a_*$ and $x_0=a_*$.

If $x_0>a_*$, by the strict monotonicity of the boundary we can also assume with no loss of generality that $v(t_0,x)=g(x)$ and $\partial_{xx}v(t_0,\ud x)=g''(\ud x)$ for $x\in(a_*,x_0)$ (notice that $g''(\ud x)$ is a continuous measure since $g\in C^1([a,b])$).
Therefore $\partial_{xx}v(t_0,\,\cdot\,)$ defines a signed measure on $(a_*,b)$ with a single atom at $x_0$. That is
\begin{align}\label{eq:vxx-meas}
\partial_{xx}v(t_0,\ud x)=\ind_{(a_*,x_0)}(x)g''(\ud x)+\theta\,\delta_{x_0}(\ud x)+\ind_{(x_0,b)}(x)\partial_{xx}v(t_0,x)\ud x.
\end{align} 
Setting $\zeta_0=\inf\{s\ge 0: X_s\notin (a_*,b)\}$, by the super-harmonic property of the value function we have 
\begin{align*}
v(t_0,x_0)\ge& \E_{x_0}\left[\e^{-\int_0^{s\wedge\zeta_0}r(X_u)\ud u}v(t_0+s\wedge\zeta_0,X_{s\wedge\zeta_0})\right]\\
\ge& -\kappa\, s+\E_{x_0}\left[\e^{-\int_0^{s\wedge\zeta_0}r(X_u)\ud u}v(t_0,X_{s\wedge\zeta_0})\right],
\end{align*}
where we used that for $x\in[a,b]$ and $s\le s_0<T-t$ there is a positive constant $\kappa=\kappa(s_0,a,b)>0$ such that $v(t+s,x)-v(t,x)\ge -\kappa s$, thanks to Assumption \ref{ass:const} (which guarantees Theorem \ref{thm:C1-time}).
Now we can use It\^o-Tanaka-Meyer formula and \eqref{eq:vxx-meas} to rewrite the term under expectation and obtain
\begin{align}\label{eq:meyer}
0\ge& \E_{x_0}\left[\int_{(a_*,x_0)}\!\!\ell^z_{s\wedge\zeta_0}\mu(\ud z)\!+\!\int_{(x_0,b)}\!\!\ell^z_{s\wedge\zeta_0}\Big(\partial_{xx}v(t_0,z)\!-\!2\sigma^{-2}(z)r(z)v(t_0,z)\Big)\ud z\right]\\
&+\theta\,\E_{x_0}[ \ell^{x_0}_{s\wedge\zeta_0}]-2\kappa\,s\notag\\
\ge &\E_{x_0}\left[\int_{(a_*,x_0)}\ell^z_{s\wedge\zeta_0}\mu(\ud z)+\theta\, \ell^{x_0}_{s\wedge\zeta_0}\right]-2\kappa\,s\notag
\end{align}
where the second inequality uses that $\partial_{xx}v-2\sigma^{-2}r v=-2\sigma^{-2}\partial_tv\ge 0$ in $\cC_{a,b}$. Since $g\in C^1(a,b)$ and $(a,b)\subset\Lambda_-$, the measure $\mu(\ud x)$ is continuous and negative on $(a,b)$. Then, the same estimates as in the proof of Lemma \ref{lem:tech} (or Proposition \ref{prop:bay}) allow us to conclude that for $s>0$ sufficiently small we reach the contradiction 
\[
0\ge\E_{x_0}\left[\int_{(a,x_0)}\ell^z_{s\wedge\zeta_0}\mu(\ud z)+\theta\, \ell^{x_0}_{s\wedge\zeta_0}\right]-2\kappa\,s>0.
\]
Hence $\partial_{x}v(t_0,x_0+)-g'(x_0)=0$ as claimed.

If $x_0=a_*$, either there is $\delta>0$ such that $\{t_0\}\times(a_*-\delta,a_*]\in\cS$ or $\lim_{x\uparrow a_*}c(x)>t_0$ (recall that $x\mapsto c(x)$ is strictly decreasing on $(a,a_*]$). In the former case, we can repeat the same arguments as for the case $x_0>a_*$ but considering the interval $(a_*-\delta,b)$ and the stopping time $\zeta_0=\inf\{s\ge 0: X_s\notin (a_*-\delta,b)\}$ in \eqref{eq:meyer}. In the latter case instead $\{t_0\}\times\big[(a,a_*)\cup(a_*,b)\big] \in\cC$. Then $x\mapsto \partial_{xx}v(t_0,x)$ is continuous (hence bounded) on $[a',a_*]\cup[a_*,b']$ by \eqref{eq:vxx}, with a single jump at $x=a_*=x_0$. Since $v(t_0,x_0)=g(x_0)$, $\partial_x v(t_0,x_0-)\le g'(x_0)$ and we are assuming $\partial_x v(t_0,x_0+)- g'(x_0)>0$, then $\theta':=\partial_x v(t_0,x_0+)-\partial_x v(t_0,x_0-)>0$. We can define the signed measure
\begin{align}\label{eq:vxx-meas2}
\partial_{xx}v(t_0,\ud x)=\theta'\,\delta_{x_0}(\ud x)+\ind_{\{x\neq x_0\}}\partial_{xx}v(t_0,x)\ud x,\quad\text{for $x\in(a,b)$}
\end{align} 
and argue in a similar way as in \eqref{eq:meyer}. We then obtain the contradiction
\[
0\ge\E_{x_0}\left[\int_{(a',b')\setminus\{x_0\}}\!\!\ell^z_{s\wedge\tau_0}\Big(\partial_{xx}v(t_0,z)\!-\!2\sigma^{-2}(z)r(z)v(t_0,z)\Big)\ud z+\theta'\ell^{x_0}_{s\wedge\tau_0}\right]\!-\!2\kappa\,s>0,
\]
with $\tau_0$ as in \eqref{eq:tau0}. Hence $\partial_{x}v(t_0,x_0+)=\partial_{x}v(t_0,x_0-)=g'(x_0)$ as claimed.
\hfill$\square$

\end{document}